\documentclass[11pt,letterpaper]{article}
\usepackage[margin=1in]{geometry}
\usepackage{lipsum}
\usepackage{amsfonts}
\usepackage{graphicx}
\usepackage{epstopdf}
\usepackage{amsmath}
\ifpdf
  \DeclareGraphicsExtensions{.eps,.pdf,.png,.jpg}
\else
  \DeclareGraphicsExtensions{.eps}
\fi

\usepackage{enumerate}
\usepackage{amsthm}
\theoremstyle{definition}

\newtheorem{condition}{Condition}
\newtheorem{example}{Example}
\newtheorem{theorem}{Theorem}
\newtheorem{lemma}{Lemma}
\newtheorem{corollary}{Corollary}

\newtheorem{remark}{Remark}

\usepackage{mathtools}
\usepackage[noabbrev,capitalise]{cleveref}
\crefname{condition}{Condition}{Condition}
\usepackage{algorithm}
\usepackage{algorithmic}

\usepackage{enumerate}
\usepackage{comment}

\DeclareMathOperator*{\argmin}{arg\,min}

\title{Sequential Convex Restriction and its Applications in Robust Optimization}
\author{Dongchan Lee\thanks{Department of Mechanical Engineering, Massachusetts institute of Technology, Cambridge, MA 02139 ({dclee@mit.edu, jjs@mit.edu}).}
\and Konstantin Turitsyn\thanks{ D. E. Shaw Group, New York, NY 10036 ({turitsyn@mit.com}).}
\and Jean-Jacques Slotine\footnotemark[1]}
\date{}
\begin{document}

\maketitle

\begin{abstract}
This paper presents a convex sufficient condition for solving a system of nonlinear equations under parametric changes and proposes a sequential convex optimization method for solving robust optimization problems with nonlinear equality constraints.
By bounding the nonlinearity with concave envelopes and using Brouwer's fixed point theorem, the sufficient condition is expressed in terms of closed-form convex inequality constraints.
We extend the result to provide a convex sufficient condition for feasibility under bounded uncertainty.
Using these conditions, a non-convex optimization problem can be solved as a sequence of convex optimization problems, with feasibility and robustness guarantees.
We present a detailed analysis of the performance and complexity of the proposed condition.
The examples in polynomial optimization and nonlinear network are provided to illustrate the proposed method.
\end{abstract}

\section{Introduction}
Nonlinear equations are ubiquitous in modeling optimization problems, but they impose unique challenges in ensuring solvability and feasibility of the solution.
This paper presents a method for guaranteeing solvability and feasibility under uncertainty and shows how it can be applied to solve a robust optimization problem subject to nonlinear equality constraints:
\begin{equation}
	\begin{aligned}
		\underset{u,x}{\text{minimize}} \hskip 1em & f_0(u) \\
		\text{subject to} \hskip 1em & f(x,u,w)= 0, \\
		& h(x,u,w)\leq 0, \hskip 1em \forall w\in\mathcal{W},
	\end{aligned}
	\label{eqn:robustOPT_orig}
\end{equation}
where $f:(\mathbf{R}^n,\mathbf{R}^m,\mathbf{R}^r)\rightarrow \mathbf{R}^n$ and $h:(\mathbf{R}^n,\mathbf{R}^m,\mathbf{R}^r)\rightarrow \mathbf{R}^s$ are vectors of continuous nonlinear functions. The decision variables are divided into $x\in\mathbf{R}^n$, referred to as implicit (decision) variables, and $u\in\mathbf{R}^m$, referred to as explicit (decision) variables. Explicit variables are a subset of decision variables that are independent of the uncertain variables, and implicit variables are a subset of decision variables that adapt to the uncertain variables according to the equality constraints. Note that the number of equality constraints and the number of implicit variables are the same, so the implicit variables can be solved by the system of equations if explicit variables are appropriately chosen.
Uncertain variables are denoted by $w\in\mathbf{R}^r$ and are restricted to the uncertainty set, $\mathcal{W}$. 
The objective function is $f_0:\mathbf{R}^m\rightarrow \mathbf{R}$ and is a convex function of $u$ without loss of generality. If the objective does not meet this condition, it can be rewritten in this form by adding $h_{s+1}(x,u,w)=f_0(x,u)-u_{m+1}$ and setting $u_{m+1}$ as the objective function.
This is a semi-infinite optimization problem, where the constraints need to be satisfied for all realizations of the uncertainty set. This paper shows classes of sparse nonlinear constraints and uncertainty sets where this problem can be solved efficiently by a sequence of convex optimization problems.

Before discussing the constraints under uncertainty, we first consider the problem without uncertainty where the uncertainty set is a singleton, $\mathcal{W}=\{w^{(0)}\}$. We will refer to $w^{(0)}$ as the nominal uncertain variable, and the constraints in this case will be referred to as the nominal constraint. When the nominal constraint is considered, the argument $w$ will be dropped and we will simply write
\begin{subequations}\begin{align}
f(x,u)&=0 \label{eqn:eq}\\
h(x,u)&\leq 0, \label{eqn:ineq}
\end{align}\label{eqn:the_eqn}\end{subequations}
as the constraint for the problem. The feasible domain of explicit variables satisfying the nominal constraints is denoted by
\begin{displaymath}
\mathcal{U}=\{u\mid \exists\, x,\ f(x,u)=0,\ h(x,u)\leq 0  \}.
\end{displaymath}
This notion implies that the solution manifold satisfying $f(x,u)=0$ is projected onto the space of explicit variables. As an example, consider a quadratic equation, $f(x,u)=x^2+u_1x+u_2$. The projection of the manifold leads to the well-known solvability condition, $\mathcal{U}=\{u\mid u_1^2-4u_2\geq0\}$. The illustration of both manifold and its projection is shown in \cref{fig:quad_prj}. Notice that  (a) finding a general solvability condition for a large system of nonlinear equations is generally difficult if possible, and (b) the solvability condition forms a non-convex set.

\begin{figure}[!htbp]
	\centering
	\includegraphics[width=2.2in]{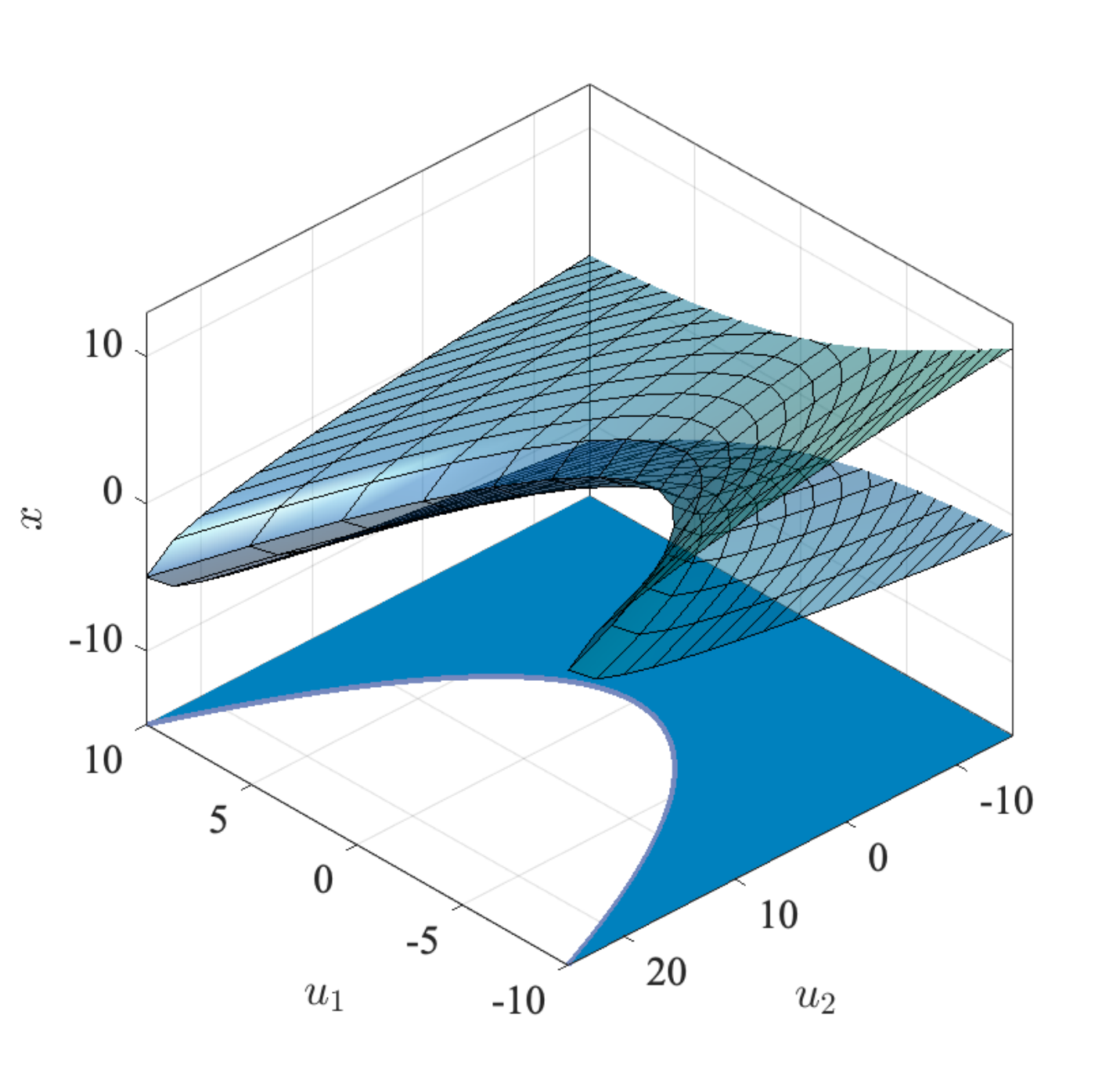}
	\caption{Projection of the manifold created by $x^2+u_1x+u_2=0$ onto the explicit variable space.}
	\label{fig:quad_prj}
\end{figure}

This paper considers the \textit{convex restriction} of $\mathcal{U}$, which we denote by $\mathcal{U}^\textrm{cvxrs}\subseteq\mathcal{U}$.
The convex restriction provides a convex sufficient condition for the feasibility of the explicit variable $u$, and can be written with a closed-form expression based on the envelope over the nonlinear functions.
We show extended analysis of the convex restriction and demonstrate its application to solve the robust optimization problem with feasibility guarantee and the bound on the optimality gap. It may not be obvious at first how the convex restriction is useful for solving the robust optimization problem, but it turns out to be a simple extension.

The paper has the following structure. 
Section \ref{sec:prelim} will show the representation of the constraints and preliminaries.
In Section \ref{sec:cvxrs}, we discuss convex restriction as originally proposed in \cite{Lee2018}, and provide extended analysis and properties. Namely, we will show (a) the explicit number of constraints involved in convex restriction, (b) the retrieval of implicit variables, (c) the non-emptiness of convex restriction around its neighborhood, (d) the equivalence when the constraints are convex, and (e) the complexity and performance trade-off based on the representation. 
In Section \ref{sec:rscrs}, we will extend the convex restriction to include bounded uncertain variables and derive a sufficient condition for robust feasibility of constrain¥ts in \cref{eqn:robustOPT_orig}.
Section \ref{sec:scrs} will study the sequential convex restriction, which iteratively solves convex optimization problems by replacing the original constraints with convex restriction conditions. We will show that the algorithm for the nominal constraint converges to a point satisfying the Karush-Kuhn-Tucker (KKT) condition. Section \ref{sec:conclusion} provides some concluding remarks.

\section{Preliminaries}\label{sec:prelim}

\subsection{Decomposed Representation}
The constraints in \cref{eqn:eq} and \cref{eqn:ineq} can be represented as linear combinations of continuous basis functions,
\begin{equation} \begin{aligned}
f(x,u)&=M\psi(z,u) \\
h(x,u)&=L\psi(z,u) \\
z&=Cx,
\end{aligned}\label{eqn:decomposed}\end{equation}
where $\psi:(\mathbf{R}^q,\mathbf{R}^m)\rightarrow \mathbf{R}^p$ is a vector of nonlinear basis functions, and $M\in\mathbf{R}^{n\times p},\, L\in\mathbf{R}^{s\times p}$ and $C\in\mathbf{R}^{q\times n}$ are constant matrices. The variable $z\in\mathbf{R}^q$ is a linearly transformed implicit variable and is assumed to satisfy the following condition.

\begin{condition}
$\mathbf{rank}(C)=n$. Equivalently, $\mathcal{P}=\{x\mid z=Cx,\ z^\ell\leq z\leq z^u\}$ is closed for some $z^u,\,z^\ell\in\mathbf{R}^{q}$.
\label{cond:closed}
\end{condition}

The representation in Equation \cref{eqn:decomposed} satisfying \cref{cond:closed} always exists where a trivial example is setting $M$ and $C$ to be the identity matrix, and $\psi(z,u)=f(x,u)$ with $z=x$. The set of basis functions is not unique, and there is a natural trade-off between the complexity and conservatism based on the choice of the basis functions (see \cref{ex:consv_vs_cplx}). In addition, the transformed implicit variable $z$ needs to be chosen such that $\psi_i$ is a function of only a finite subset of $\{z_1,...,z_q\}$. To make this statement more precise, let $\mathcal{I}_k$ denote the set of indices of $z$ that the basis function $\psi_k$ depends on. That is, given $e_j\in\mathbf{R}^q$ is a unit vector with $j$th element equal to 1 and zero otherwise,
\begin{displaymath}
\mathcal{I}_k=\left\{j \ \middle\vert\ \exists\, (z,\,u,\,\varepsilon\neq0),\ \ \psi_k(z,u)\neq\psi_k(z+\varepsilon e_j,u)\right\}.
\end{displaymath}
The degree of sparsity of the representation is defined as the worst-case cardinality of $\mathcal{I}_k$ and is denoted by $\lvert\mathcal{I}\rvert$ where
\begin{displaymath}
\lvert\mathcal{I}\rvert=\max_{k\in\{1,...,p\}}{\lvert \mathcal{I}_k\rvert}.
\end{displaymath}
It will be shown later that the number of constraints involved in the convex restriction grows exponentially with respect to $\lvert\mathcal{I}\rvert$, but there often exists a natural choice of $z$ such that $\lvert\mathcal{I}\rvert$ does not grow with respect to the size of the problem. The following example in a network flow problem shows how these variables can be chosen.

\begin{example}{(Nonlinear Network Flow Problem)}\label{ex:oscillators}
Consider a directed graph $G=(\mathcal{N},\mathcal{A})$ with $\theta_i$ and $b_i$ representing the internal state and external supply at each node $i\in\mathcal{N}$, and $E$ denoting the incidence matrix of the graph. 
Suppose the flow model between node $i$ and $j$ is given by a nonlinear function $\sigma:\mathbf{R}\rightarrow\mathbf{R}$. The conservation of the flows at every node imposes the constraint,
\begin{displaymath}
b_i+\sum_{j\in I(i)}\sigma(\theta_j-\theta_i)=\sum_{j\in O(i)}\sigma(\theta_i-\theta_j), \hskip2em \forall i\in\mathcal{N},
\end{displaymath}
where $I(i)$ is the set of start nodes of the edges that are incoming to, and $O(i)$ is the set of end nodes of the edges that are outgoing from, node $i$. 
Suppose that the supply $b_i$ at node $i=\{2,\ldots,\lvert\mathcal{N}\rvert\}$ are controlled while $b_1$ balances the overall supply and demand. Then the explicit variables are $u=\begin{bmatrix}b_2 & \ldots & b_{\lvert\mathcal{N}\rvert}\end{bmatrix}^T$, and the implicit variables are $x=\begin{bmatrix}\theta^T & b_1\end{bmatrix}^T$.
Let the transformed variable be $z=\begin{bmatrix}E^T\theta & b_1 \end{bmatrix}$ by choosing $C=\mathbf{blkdiag}(E^T,1)$.
The equality constraint can be represented by $M=\begin{bmatrix}-E & I\end{bmatrix}$ with the basis function  $\psi(z,u)=\begin{bmatrix} \sigma (z_1) & ... & \sigma (z_{\lvert\mathcal{A}\rvert}) & z_{\lvert\mathcal{A}\rvert+1} & u^T\end{bmatrix}^T$. 
Since $\psi_i$ has only one variable as the argument for all $i$, the degree of sparsity is $\lvert \mathcal{I}\rvert=1$ independent of the size of the network.
\end{example}

An important feature to notice is that the nonlinearity of $\psi_k$ can be arbitrarily bounded by constraining only $\lvert\mathcal{I}_k\rvert$ variables.
\begin{lemma}
For all $u\in\mathcal{U}$ and $\varepsilon>0$, there exists some $\delta$ such that if $\mathcal{P}_k=\{x\mid z=Cx,\ z_i^\ell\leq z_i\leq z_i^u,\,\forall\,i\in\mathcal{I}_k\}$ with $z^u_i-z^\ell_i<\delta$ for all $i\in\mathcal{I}_k$, then 
\begin{displaymath}\begin{aligned}
\max_{x\in\mathcal{P}_k}{\psi_k(Cx,u)}-\min_{x\in\mathcal{P}_k}{\psi_k(Cx,u)}&<\varepsilon.
\end{aligned}\end{displaymath}
\label{lemma:cont}
\end{lemma}
\begin{proof}
Suppose $\mathcal{P}=\{x\mid z=Cx,\ z^\ell\leq z\leq z^u\}$ with  $z^u-z^\ell<\delta$. 
Since the basis functions are continuous, for all $u\in\mathcal{U}$ and $\varepsilon>0$, there exists $\delta$ such that if $z^u-z^\ell<\delta$ then 
\begin{displaymath}
\max_{x\in\mathcal{P}}{\psi_k(Cx,u)}-\min_{x\in\mathcal{P}}{\psi_k(Cx,u)}<\varepsilon.
\end{displaymath}
Since the $\psi_k(x,u)$ is independent of $z_j$ with $j\in\{1,...,q\}\setminus\mathcal{I}_k$, 
\begin{displaymath}
\max_{x\in\mathcal{P}_k}{\psi_k(Cx,u)}-\min_{x\in\mathcal{P}_k}{\psi_k(Cx,u)}=\max_{x\in\mathcal{P}}{\psi_k(Cx,u)}-\min_{x\in\mathcal{P}}{\psi_k(Cx,u)}<\varepsilon.
\end{displaymath}
\end{proof}

The effect of nonlinearity can be controlled by bounding a finite number of variables as we saw in \cref{ex:oscillators}.
This property of the sparse representation will drastically reduce the complexity involved in convex restriction.

\subsection{Brouwer's Fixed Point Theorem}
Brouwer's fixed point theorem has been widely used in game theory, economics and in dynamical systems \cite{brouwer1911abbildung,border1989fixed,florenzano2003general}.
In this paper, Brouwer's fixed point theorem will be used in the proof to certify the existence of the implicit variable that satisfies the given constraints.
\begin{theorem} (Brouwer's Fixed Point Theorem) 
Let $\mathcal{P}\subseteq\mathbf{R}^n$ be a nonempty compact convex set and $G:\mathcal{P}\rightarrow\mathcal{P}$ be a continuous mapping. Then there exists some $x\in\mathcal{P}$ such that $G(x)=x$.
\label{thm:Brouwer}
\end{theorem}

The convex restriction will be derived by designing the fixed-point equation from Newton's iteration and the self-mapping set $\mathcal{P}$ to be a polytope that is parametrized by its affine term. Using the sparse representation of the constraints, we will show that the number of constraints in convex restriction is linearly proportional to the number of constraints of the original problem.

\section{Convex Restriction}\label{sec:cvxrs}
The convex restriction provides an analytical expression for a convex sufficient condition for feasibility in \cref{eqn:the_eqn} around some nominal point, $(x^{(0)},\,u^{(0)})$. While convex relaxation has a globally optimal outer-approximation, which is the convex hull of the feasible set, the convex restriction can have many local regions where it cannot form a larger region due to the restriction as a convex set. We use the nominal point as the reference point around which the convex restriction is constructed. The nominal point is assumed to satisfy the following conditions.
\begin{condition} The nominal point, $(x^{(0)},\,u^{(0)})$, satisfies
\begin{enumerate}[(i)]
\item $f(x^{(0)},\,u^{(0)})=0$ and $h(x^{(0)},\,u^{(0)})\leq0$, and
\item If $f(x,u)$ is differentiable with respect to $x$, $\nabla_x f(x,u)\mid_{(x^{(0)},u^{(0)})}$ is invertable.
\end{enumerate}
\label{cond:nominal_point}
\end{condition}

\cref{cond:nominal_point} is not strictly necessary in constructing convex restriction, but it will be used later in the analysis of the algorithm proposed based on the convex restriction.
From the Implicit Function Theorem, it is known that there is a neighborhood of solutions where $x$ can be expressed as a function of $u$ if \cref{cond:nominal_point} is satisfied. The convex restriction will provide the bounds on the implicit variable and an explicit description of a convex neighborhood where the existence of the implicit variable is guaranteed.

\subsection{Fixed Point Representation}\label{subsec:Fixed_pt_eqn}
Here we present the fixed point representation of the equality constraint $f(x,u)=0$. The equality constraint can be rewritten in the following fixed point form,
\begin{equation}
    x=-(M\Lambda C)^{-1}Mg(z,u),
    \label{eqn:fxdpt_form}
\end{equation}
where 
\begin{displaymath}
    g(z,u)=\psi(z,u)-\Lambda z,
\end{displaymath}
with some matrix $\Lambda\in\mathbf{R}^{p\times q}$. The conservatism of the convex restriction depends on the choice of $\Lambda$. Finding the optimal $\Lambda$ that maximizes the region for convex restriction is difficult, but the Jacobian evaluated at the base point gives a good approximate solution.

\begin{itemize}
\item If $f$ is differentiable at the nominal point and $\nabla_xf(x_0,u_0)$ is non-singular, choose $\Lambda$ as the Jacobian of the basis function with respect to $z$ evaluated at the base point,
\begin{displaymath}
\Lambda=\nabla_z \psi(z,u^{(0)})\mid_{z=z^{(0)}}.
\end{displaymath}
Note that in this case $M\Lambda C=\nabla_x f(x,u^{(0)})\mid_{x=x^{(0)}}$.

\item If $f$ is non-differentiable at the nominal point, choose each element of $\Lambda$ as 
\begin{displaymath}
\Lambda_{ij}=\partial_{z_j} \psi_i(z,u^{(0)})\mid_{z=z^{(0)}},
\end{displaymath}
where $\partial_{z_j} \psi_i(z,u)$ is the subgradient of $\psi_i$ if $\psi_i$ is locally convex with respect to $z_j$ at the nominal point. If $\psi_i$ is locally concave, then $\partial_{z_j} \psi_i(z,u)$ is the supergradient.
\end{itemize}
For differentiable functions, the fixed point form in \cref{eqn:fxdpt_form} is equivalent to a single step of Newton's method, $x=-J_f^{-1}\left(f(x,u)-J_fx\right)$, where $J_f=\nabla_x f(x,u)$.

Given the explicit variable $u$, Equation \cref{eqn:fxdpt_form} defines a continuous nonlinear operator $G:\mathbf{R}^n\rightarrow\mathbf{R}^n$ that maps the implicit variable $x$ to $-(M\Lambda C)^{-1}Mg(Cx,u)$. By iterating this operator, a sequence of approximate solutions can be generated with the initial condition $x=x^{(0)}$ for an arbitrary value of $u$. We will verify the existence of the implicit variable by studying this sequence of approximate solutions and inferring the existence of a fixed point of the sequence. 

\subsection{Self-mapping Polytope}
Consider the following set of polytopes as a candidate for the self-mapping set in Brouwer's fixed point theorem,
\begin{displaymath}
\mathcal{P}(b)=\{x\mid Ax\leq b\} \\
\end{displaymath}
where
\begin{equation}\begin{aligned}
A=\begin{bmatrix} C \\ -C \end{bmatrix}, \ b=\begin{bmatrix} z^u \\ -z^\ell \end{bmatrix}.
\end{aligned}\label{eqn:def_Ab}\end{equation}
An alternative representation of this polytope is
\begin{displaymath}
\mathcal{P}(b)=\{x\mid z=Cx,\ z^\ell\leq z\leq z^u\}.
\end{displaymath}
Since $\mathcal{P}(b)$ is a polytope satisfying \cref{cond:closed}, the set is compact and convex. The set of polytopes parameterized by $b\in\mathbf{R}^{2q}$ will be used to guarantee the existence of an implicit solution using Brouwer's fixed point theorem.

\begin{lemma}
For a given explicit variable $u$, there exists an implicit variable $x$ that satifies $f(x,u)=0$ if and only if there exists $b\in\mathbf{R}^{2q}$ such that
\begin{equation}
\max_{x\in\mathcal{P}(b)} K_ig(Cx,u)\leq b_i,\hskip2em i=1,\ldots,2q,
\label{eqn:iff_solvable}
\end{equation}
where $K_i\in\mathbf{R}^{1\times p}$ is the $i$th row of matrix $K$ and
\begin{displaymath}
K=-A(M\Lambda C)^{-1}M.
\end{displaymath}
\label{lemma:iff_solvable}
\end{lemma}
\begin{proof} The condition in \cref{eqn:iff_solvable} implies that $-(M\Lambda C)^{-1}Mg(Cx,u)\in\mathcal{P}(b)$ for all $x\in\mathcal{P}(b)$. Then the set $\mathcal{P}(b)$ is self-mapping with the nonlinear map $G(x)=-(M\Lambda C)^{-1}Mg(Cx,u)$, so there exists an implicit variable $x\in\mathcal{P}(b)$ from Brouwer's fixed point theorem. To prove that this is a necessary condition, suppose there exists $(x^{(0)},u^{(0)})$ satisfying \cref{eqn:eq}. Choose $b=[(Cx^{(0)})^T \, (-Cx^{(0)})^T]^T$, then for $i=1,\ldots,2q$, $\max_{x\in\mathcal{P}(b)} K_ig(x,u)=A_ix^{(0)}=b_i$, which satisfies the condition \cref{eqn:iff_solvable}.
\end{proof}

\subsection{Concave Envelopes}
Suppose that the function $g_k$ is known to be bounded by some analytical functions $g_k^u$ and  $g_k^\ell$ such that
\begin{equation*}
g_k^\ell(z,u)\leq g_k(z,u)\leq g_k^u(z,u),
\end{equation*}
where the envelopes satisfy the following conditions.
\begin{condition} $g_k^u$ and $g_k^\ell$ are over- and under-estimators of $g_k$ such that
\begin{enumerate}[(i)]
\item $g_k^u$ is convex and $g_k^\ell$ is concave function of $z$ and $u$,
\item $g_k^u$ and $g_k^\ell$ are tight at the nominal point, 
\begin{displaymath}g_k^\ell(z^{(0)},u^{(0)})=g_k(z^{(0)},u^{(0)})=g_k^u(z^{(0)},u^{(0)}),\end{displaymath}
\item if $g$ is differentiable at the nominal point, the derivatives of estimators are tight,
\begin{displaymath}\begin{aligned}
\nabla_z g_k^\ell(z,u^{(0)})\big\lvert_{z=z^{(0)}}&=\nabla_z g_k^u(z,u^{(0)})\big\lvert_{z=z^{(0)}},\\
\nabla_u g_k^\ell(z^{(0)},u)\big\lvert_{u=u^{(0)}}&=\nabla_u g_k^u(z^{(0)},u)\big\lvert_{u=u^{(0)}}.
\end{aligned}\end{displaymath}
\end{enumerate}
\label{cond:env}
\end{condition}
Similarly $\psi_k^u$ and $\psi_k^\ell$ are the over- and under-estimators of $\psi_k$ that satisfy \cref{cond:env}. We define the envelope such that the nonlinear function is bounded by a convex over-estimator and a concave under-estimator as a \textit{concave envelope}. For any continuous function, there exists a concave envelope satisfying \cref{cond:env}.

\begin{figure}[!htbp]
	\centering
	\includegraphics[width=2.0in]{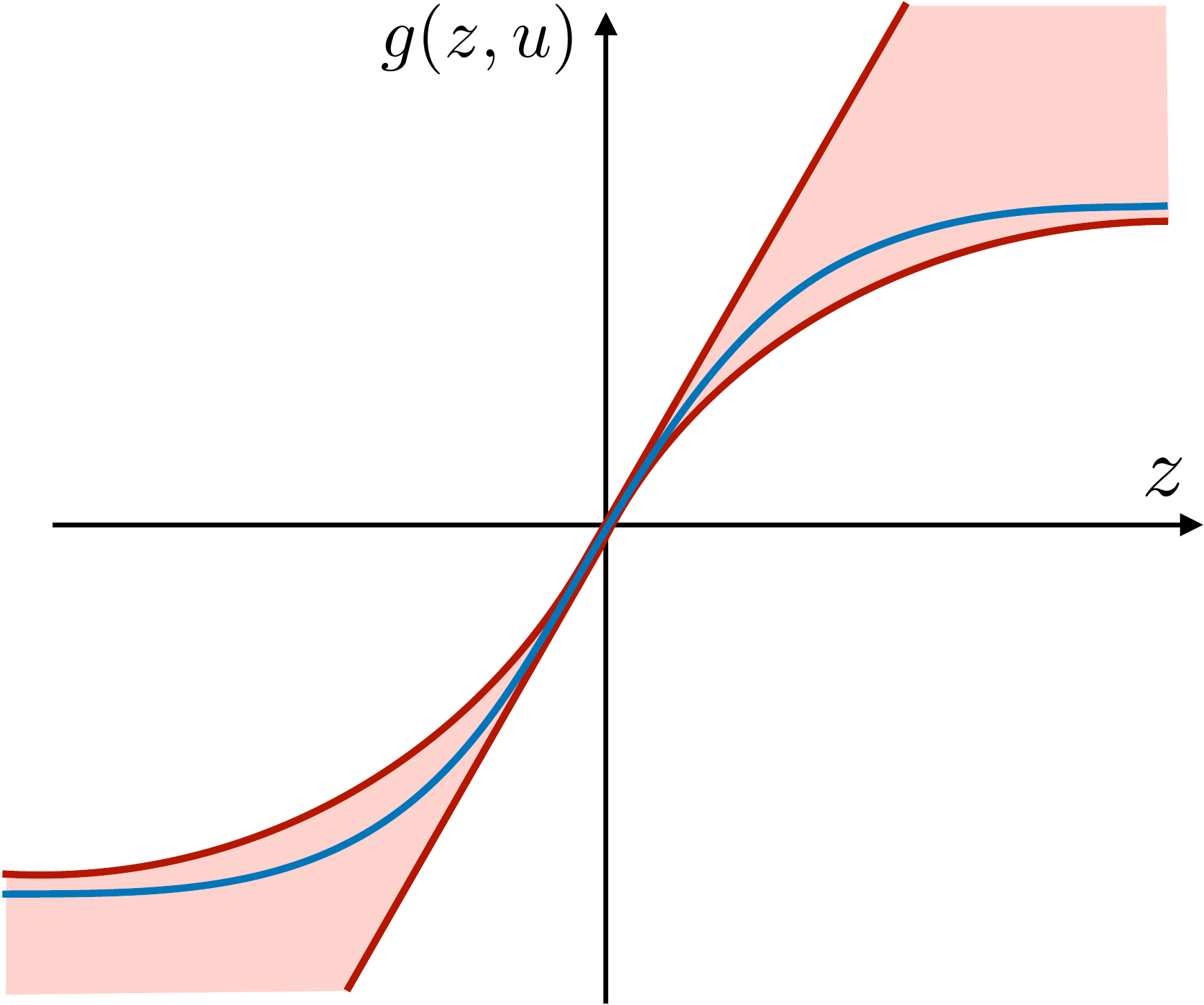}
	\caption{Illustration of a concave envelope.}
	\label{fig:convenv_ex}
\end{figure}

An example of such an envelope is shown in \cref{fig:convenv_ex}, which turns out to be necessary for enforcing convexity to a restricted set. These envelopes have flipped convexity and concavity compared to the envelopes used in convex relaxation \cite{Hijazi2017,Mitsos}. Given the model of the system, these envelopes are assumed to have a closed-form expression, and we will discuss the derivation of concave envelopes in the next section.

\subsubsection{Quadratic Concave Envelopes}
Concave envelopes that satisfy \cref{cond:env} can be obtained systematically based on the Taylor series of $g_k$. To make the notation more compact, let $y=\begin{bmatrix} z^T & u^T \end{bmatrix}^T$ and $y_0$ be the nominal value. The Taylor series of $g_k(y)$ is
\begin{displaymath}
g_k(y)=g_k(y^{(0)})+\nabla_y g_k(y-y^{(0)})+(y-y^{(0)})^TH(y)(y-y^{(0)})+\tau(y),
\end{displaymath}
where $H(y)\in\mathbf{R}^{(q+m)\times(q+m)}$ is the Hessian of $g_k$ evaluated at $y$, and $\tau(y)$ is the residual term of polynomial order greater than 2. Suppose the residual term can be bounded by
\begin{displaymath}
(y-y^{(0)})^T(Q^\ell_k-H(y))(y-y^{(0)})\leq\tau(y)\leq (y-y^{(0)})^T(Q^u_k-H(y))(y-y^{(0)}),
\end{displaymath}
where $Q^\ell,\,Q^u\in\mathbf{R}^{(q+m)\times(q+m)}$ are constant negative semi-definite and positive semi-definite matrices, respectively. If $g_k(y)$ is continuous and has a scalar argument, $Q^u$ and $Q^\ell$ can be computed by $Q^u=\sup_y \lvert \frac{d^2}{dy^2}g_k(y)\rvert$ and $Q^\ell=-Q^u$. For multi-variable functions, there are typically multiple $Q^u$ and $Q^\ell$ to choose from. (See bilinear function example with a parameter $\rho$ in \cref{apdx:bilinear}.)
Given the Taylor series, the quadratic concave envelope can be written as
\begin{equation}\begin{aligned}
g_k^u(z,u)&=g_k^{(0)}+\nabla g_k^{(0)}\begin{bmatrix}z-z^{(0)} \\ u-u^{(0)}\end{bmatrix}+\begin{bmatrix}z-z^{(0)} \\ u-u^{(0)}\end{bmatrix}^TQ_k^u\begin{bmatrix}z-z^{(0)} \\ u-u^{(0)}\end{bmatrix} \\
g_k^\ell(z,u)&=g_k^{(0)}+\nabla g_k^{(0)}\begin{bmatrix}z-z^{(0)} \\ u-u^{(0)}\end{bmatrix}+\begin{bmatrix}z-z^{(0)} \\ u-u^{(0)}\end{bmatrix}^TQ_k^\ell\begin{bmatrix}z-z^{(0)} \\ u-u^{(0)}\end{bmatrix}
\end{aligned}\end{equation}
where $g_k^{(0)}=g_k(z^{(0)},u^{(0)})$ and $\nabla g_k^{(0)}=\begin{bmatrix}\nabla_z g_k(z^{(0)},u^{(0)}) & \nabla_u g_k(z^{(0)},u^{(0)})\end{bmatrix}$. The envelopes derived with this procedure satisfy \cref{cond:env} with the nominal point at $(x^{(0)},u^{(0)})$. Examples of quadratic envelopes are provided in \cref{apdx:concave_env} for bilinear, trigonometric, and logistic functions, which were derived with the proposed procedure. When the quadratic envelopes are used, the resulting convex restriction will be convex quadratic inequality constraints.

\subsection{Bounds over Intervals}
Given the concave envelopes, the bound of $g_k$ over the polytope $\mathcal{P}(b)$ can be defined as
\begin{equation*}
g^\ell_{\mathcal{P},k}(u,b)\leq g_k(z,u)\leq g^u_{\mathcal{P},k}(u,b),
\end{equation*}
which is valid for all $z\in\{Cx\mid x\in\mathcal{P}(b)\}$. These bounds are defined as
\begin{equation*}
\begin{aligned}
g^u_{\mathcal{P},k}(u,b)&=\max_{x\in\mathcal{P}(b)}g^u_k(Cx,u) \\
g^\ell_{\mathcal{P},k}(u,b)&=\min_{x\in\mathcal{P}(b)}g^\ell_k(Cx,u).
\end{aligned}
\end{equation*}
Since $g^u_k(z,u)$ is a convex function, its maximum occurs at at least one of the vertices of the polytope $\mathcal{P}(b)$. Similarly, the minimum of concave $g^\ell_k(z,u)$ occurs at the vertex. The self-mapping condition in Brouwer's fixed point theorem can be viewed as solving a containment of the polytope $\mathcal{P}(b)$ into the inequality constrained sets \cref{eqn:iff_solvable,eqn:ineq}. Solving the containment problem is generally hard, but it becomes tractable if the polytope is in the vertex representation contained in a convex set \cite{Kellner2013}.
\subsection{Vertex Tracking}
By relaxing the equations with concave envelopes, the interval bound of $g_k(z,u)$ can be expressed by tracking all the vertices of the polytope
\begin{displaymath}\begin{aligned}
g^u_{\mathcal{P},k}(u,b)&=\max_{x\in\partial\mathcal{P}(b)}g^u_k(Cx,u) \\
g^\ell_{\mathcal{P},k}(u,b)&=\min_{x\in\partial\mathcal{P}(b)}g^\ell_k(Cx,u),
\end{aligned}\end{displaymath}
where $\partial\mathcal{P}(b)$ are the vertices of the polytope $\mathcal{P}(b)$. Although the number of vertices of the face-polytope $\mathcal{P}(b)$ grows exponentially with respect to the number of faces, the following lemma shows that only the vertices involved in $\mathcal{I}(k)$ need to be tracked.
\begin{lemma}
The interval bounds can be expressed with the inequalities
\begin{equation}\begin{aligned}
g^u_{\mathcal{P},k}(u,b)&\geq g^u_k(z,u), & \forall z&\in\{z\mid z_i\in\{z^\ell_i,z^u_i\},\, \forall i\in \mathcal{I}_k\} \\
g^\ell_{\mathcal{P},k}(u,b)&\leq g^\ell_k(z,u), & \forall z&\in\{z\mid z_i\in\{z^\ell_i,z^u_i\},\, \forall i\in \mathcal{I}_k\},
\end{aligned}\label{eqn:g_uell}\end{equation}
where these inequalities can be expressed by $2^{\mathcal{I}_k+1}$ inequalities by listing all possible vertices.
\label{lemma:vertex_tracking}
\end{lemma}

\cref{lemma:cont} from the previous section showed that the nonlinearity can be bounded by controlling $\lvert\mathcal{I}_k\rvert$ variables. Similarly, $g^u_{\mathcal{P},k}$ and $g^\ell_{\mathcal{P},k}$ can be expressed by inequalities involving $\lvert\mathcal{I}_k\rvert$ variables.
If the nonlinearity is decomposed in a way such that $\lvert\mathcal{I}_k\rvert$ does not grow with the problem size, the number of constraints involved is also independent of the problem size.

\subsection{Vertex Pruning}
It is not necessary to track all the vertices in \cref{eqn:g_uell} because the maximum or minimum never occurs at some of those vertices. As an example, consider the bilinear function in \cref{apdx:bilinear}. The maximum always occurs at vertices $(x^u,y^u)$ or $(x^\ell,y^\ell)$, and it is unnecessary to trace $(x^u,y^\ell)$ and $(x^\ell,y^u)$. Many of the vertices can be pruned from the candidates by exploiting this property.

\subsection{Convex Restriction and its Properties}
Given these considerations, the convex restriction of feasibility set can be expressed as an explicit condition. This condition was first provided in \cite{Lee2018}.

\begin{theorem}{(Convex Restriction of Feasibility Set)} \label{thm:CVXRS}
For a given explicit variable $u$, there exists an implicit variable $x$ that satisfies $f(x,u)=0$ and $h(x,u)\leq0$ if there exists
$b\in\mathbf{R}^{2q}$ such that
\begin{subequations}\begin{align}
K^+g^u_\mathcal{P}(u,b)+K^- g^\ell_\mathcal{P}(u,b) &\leq b \label{eqn:conv_restr_eq1}
\\
L^+\psi^u_\mathcal{P}(u,b)+L^-\psi^\ell_\mathcal{P}(u,b) &\leq 0,
\label{eqn:conv_restr_eq2}
\end{align}\label{eqn:conv_restr_eq}\end{subequations}
where $K^+_{ij}=\max\{K_{ij},0\}$ and $K^-_{ij}=\min\{K_{ij},0\}$ for each element of $K$.
\label{thm_feasibility}
\end{theorem}
\begin{proof} From Condition \cref{eqn:conv_restr_eq1}, for $i=1,\ldots,2q$,
\begin{displaymath}
\begin{aligned}
    \max_{x\in\mathcal{P}(b)} K_ig(Cx,u)&\leq \max_{x\in\mathcal{P}(b)} \left(K_i^+g^u(Cx,u)+K_i^-g^\ell(Cx,u)\right) \\
    &\leq K_i^+\max_{x\in\mathcal{P}(b)}g^u(Cx,u)+K_i^-\min_{x\in\mathcal{P}(b)}g^\ell(Cx,u) \\
    &= K_i^+\max_{x\in\partial\mathcal{P}(b)}g^u(Cx,u)+K_i^-\min_{x\in\partial\mathcal{P}(b)}g^\ell(Cx,u) \\
    &=K_i^+g^u_\mathcal{P}(u,b)+K_i^- g^\ell_\mathcal{P}(u,b)\leq b_i.
\end{aligned}
\end{displaymath}
From \cref{lemma:iff_solvable}, there exists a solution for the implicit variable, $x\in\mathcal{P}(b)$.
Similarly, from the condition \cref{eqn:conv_restr_eq2},
\begin{displaymath}
\begin{aligned}
    \max_{x\in\mathcal{P}(b)} L_i\psi(x,u)&\leq L_i^+ \psi^u_\mathcal{P}(u,b)+L_i^-\psi^\ell_\mathcal{P}(u,b)\leq 0,\ \ i=1,\ldots,s,
\end{aligned}
\end{displaymath}
so for all $x\in\mathcal{P}(b)$, $L\psi(Cx,u)\leq 0$. Therefore, there exists an implicit variable satisfying $f(x,u)=0$ and $h(x,u)\leq0$. 
\end{proof}

This is a sufficient condition for the existence of a feasible implicit variable for a given explicit variable. Note that the condition in \cref{eqn:conv_restr_eq} is a convex constraint with respect to both $u$ and $b$. This region in the explicit variable space will be denoted by $\mathcal{U}^\textrm{cvxrs}_{(0)}\subseteq\mathcal{U}$ where
\begin{displaymath}\begin{aligned}
\mathcal{U}^\textrm{cvxrs}_{(0)}=\{u\mid \exists\, b,\ K^+ g^u_\mathcal{P}(u,b)+K^- g^\ell_\mathcal{P}(u,b)&\leq b,\\
L^+ \psi^u_\mathcal{P}(u,b)+L^- \psi^\ell_\mathcal{P}(u,b) &\leq 0\}.
\end{aligned}\end{displaymath}
The subscript $(0)$ denotes that $(x^{(0)},\, u^{(0)})$ is used as the nominal point for constructing the convex restriction.

\begin{example}{(Quadratic Equations)}
Consider a quadratic equation with $x\in\mathbf{R}$ parametrized by $u\in\mathbf{R}^2$ from the introduction,
\begin{displaymath}
f(x,u)=x^2+u_1x+u_2,
\end{displaymath}
where there exist real solutions for $x$ if and only if $u_1^2-4u_2\geq0$. In addition we consider the inequality constraint, $-2\leq x \leq 2$. Define $z=x$ and the basis function $\psi(z,u)=f(x,u)$. The residual function is then $g(z,u)=\langle z, z-2z^{(0)}+u_1-u_1^{(0)}\rangle+u_2$. The bilinear envelope in \cref{apdx:bilinear} can be applied to $z$ and $z-2z^{(0)}+u_1-u_1^{(0)}$ with $\rho_1=\rho_2=1$. Let the derivative of the equation with respect to $x$ evaluated at the nominal point be denoted by  $J_0=2x^{(0)}+u_1^{(0)}$. The convex restriction condition in \cref{eqn:conv_restr_eq} gives the following closed-form expression,
\begin{equation*}\begin{aligned}
\mathcal{U}^\textrm{cvxrs}_{(0)}=\{u\mid \exists\, (z^u,\, z^\ell),\  z^u\leq2,\ z^\ell\geq-2, \hskip13em&\\
-\lvert J^{-1}_0\rvert\left(x^{(0)}(u_1-u_1^{(0)})-(x^{(0)})^2-0.25(u_1-u_1^{(0)})^2+u_2\right)&\leq z^u,\\
\lvert J^{-1}_0\rvert\left(x^{(0)}(u_1-u_1^{(0)})-(x^{(0)})^2+0.25(2z^u-2x^{(0)}+u_1-u_1^{(0)})^2+u_2\right)& \leq -z^\ell ,\\
\lvert J^{-1}_0\rvert\left(x^{(0)}(u_1-u_1^{(0)})-(x^{(0)})^2+0.25(2z^\ell-2x^{(0)}+u_1-u_1^{(0)})^2+u_2\right)& \leq -z^\ell \}.
\end{aligned}\end{equation*}

\cref{fig:quad_ex} shows this region in explicit variable space where both equality and inequality constraints were considered with the nominal point at $(x^{(0)},u^{(0)})=(0,[4, 0])$. 


\begin{figure}[!htbp]
	\centering
	\includegraphics[width=4in]{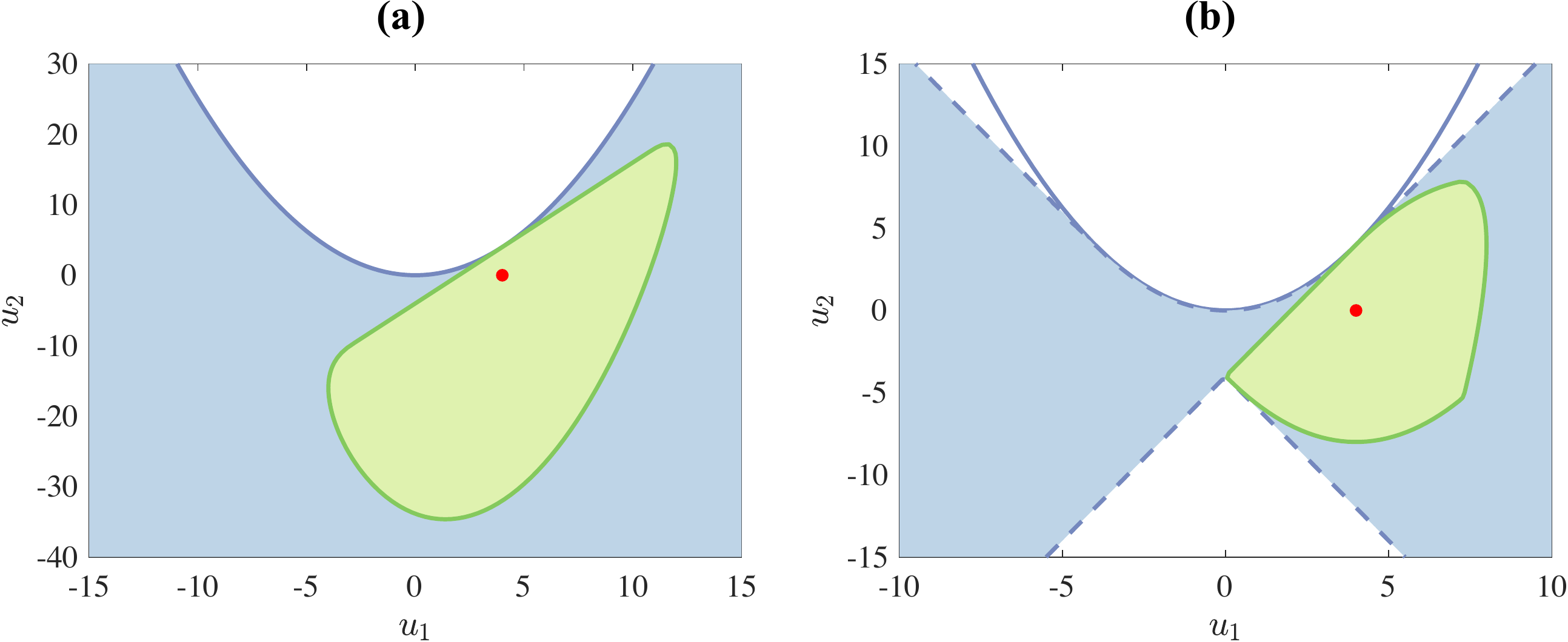}
	\caption{The convex restriction of a quadratic equation with (a) the solvability of the equality constraints and (b) the feasibility with the additional inequality constraint, $x\in[-2,2]$. The blue region shows the true feasible region, and the green region shows the convex restriction. The red dot marks the nominal point.}
	\label{fig:quad_ex}
\end{figure}
\end{example}

While the example considers a simple equation, the convex restriction creates a scalable condition for any sparse system of equations where $\lvert \mathcal{I}\rvert$ is finite, independent of the problem size.

\begin{remark}{(Scalability of Convex Restriction)} The number of constraints involved in convex restriction is bounded by
$q\cdot2^{\lvert \mathcal{I}\rvert+2}+2n+s$.
\end{remark}
There are $2n+s$ inequality constraints involved in \cref{eqn:conv_restr_eq}, and $\sum_{k=1}^q2^{\lvert \mathcal{I}(k)\rvert}$ inequality constraints involved in $g^u$, $g^\ell$, $\psi^u$ and $\psi^\ell$ as shown in \cref{eqn:g_uell}.
As we saw in \cref{ex:oscillators}, there exists a representation such that $\lvert \mathcal{I}\rvert$ is independent of the size of the original problem in many applications. Then the number of constraints involved in the convex restriction grows linearly with respect to $n$ and $s$.

\begin{remark}{(Retrieval of Implicit Variable)} \label{remark:retrieval}
Consider a sequence $\{x_k\}$ generated by $x_k=-(M\Lambda C)^{-1}Mg(x_{k-1},u)$ with $u\in\mathcal{U}^\textrm{cvxrs}$ and the initial condition, $x_0=x^{(0)}$. If the solution converges to a fixed point, $x^*$, then $f(x^*,u)=0$ and $h(x^*,u)=0$.
\label{remark:retrieval_implicit}\end{remark}

Similar to any numerical approaches for solving nonlinear equations, the above sequence is not guaranteed to converge. However, the convex restriction condition guarantees that the sequence will not diverge outside of the closed polytope $\mathcal{P}(b)$. Instead of Newton's method, the above iteration can be an alternative method to retrieve the implicit variables more efficiently without requiring the inversion of any matrix.

\begin{lemma}{(Non-emptiness of Convex Restriction)} The convex restriction is non-empty and contains the nominal point. Moreover, if there exists $b\in\mathbf{R}^{2q}$ such that
\begin{equation}\begin{aligned}
K^+g^u_\mathcal{P}(u,b)+K^- g^\ell_\mathcal{P}(u,b) &< b \\
L^+\psi^u_\mathcal{P}(u,b)+L^-\psi^\ell_\mathcal{P}(u,b) &< 0,
\end{aligned}\end{equation}
the convex restriction contains an open non-empty neighborhood around the nominal point. That is $\forall\, v\in\mathbf{R}^m,\ \exists\,\varepsilon>0$ such that $u^{(0)}+\varepsilon v\in\mathcal{U}^\textbf{cvxrs}_{(0)}$.
\label{lemma:non_empty}
\end{lemma}
\begin{proof}
Let $b^{(0)}=Ax^{(0)}$, then $\mathcal{P}(b)=\{x^{(0)}\}$ since $\mathcal{P}(b)$ is closed. Then,
\begin{equation*}
\begin{aligned}\
K^+g^u_\mathcal{P}(u^{(0)},b^{(0)})+K^- g^\ell_\mathcal{P}(u^{(0)},b^{(0)})&=Kg(x^{(0)},u^{(0)})=b^{(0)} \\
L^+\psi^u_\mathcal{P}(u^{(0)},b^{(0)})+L^-\psi^\ell_\mathcal{P}(u^{(0)},b^{(0)})&=L\psi(x^{(0)},u^{(0)})=0,
\end{aligned}
\end{equation*}
from \cref{cond:env}, so $(u^{(0)},b^{(0)})$ is always feasible to the constraints in \cref{eqn:conv_restr_eq}, and thus the convex restriction is always non-empty. Since $g^u_\mathcal{P}$ and $g^\ell_\mathcal{P}$ are convex and concave respectively, they are continuous functions with respect to $b$ and $u$. Then for all $v\in\mathbf{R}^m$, there exists $\varepsilon>0$ such that
\begin{displaymath}
\begin{aligned}
K^+ g^u_\mathcal{P}(u+\varepsilon v,b)+K^- g^\ell_\mathcal{P}(u+\varepsilon v,b)&\leq b \\
L^+ \psi^u_\mathcal{P}(u+\varepsilon v,b)+L^- \psi^\ell_\mathcal{P}(u+\varepsilon v,b)&\leq0.
\end{aligned}
\end{displaymath}
Therefore, $u+\varepsilon v\in\mathcal{U}^\textrm{cvxrs}$ from \cref{thm_feasibility}, and the convex restriction contains an open non-empty neighborhood around its nominal point.
\end{proof}

Moreover, the condition in \cref{thm:CVXRS} is an equivalent condition to the original feasibility constraints if the original constraints are convex constraints.
\begin{corollary}{(Equivalence for Convex Constraints)}
Suppose that the constraints are convex constraints: $f(x,u)$ is linear and $h(x,u)$ is convex with respect to $x$ and $u$. Then $u\in\mathcal{U}$ if and only if there exists $b\in\mathbf{R}^{2q}$ that satisfies \cref{eqn:conv_restr_eq}.
\label{corol:equiv_cvx}
\end{corollary}
\begin{proof}
Consider the decomposed representation of the constraints using the basis function $\psi(z,u)=\begin{bmatrix}f(z,u)^T & h(z,u)^T\end{bmatrix}^T$ with $z=x$, and $h^u(x,u)=h(x,u)$ since $h$ is already a convex function. Following the convex restriction procedure, the resulting condition \cref{eqn:conv_restr_eq} can be written as
\begin{equation}
z^\ell\leq-J_f^{-1}f(0,u)\leq z^u \hskip1em \text{and} \hskip1em h_k(z,u)\leq 0,\ \forall z\in\{z\mid z_i\in\{z^\ell_i, z^u_i\},\,i\in\mathcal{I}_k\},
\label{eqn:cvx_equiv}\end{equation}
for $k=1,...,s$.
From \cref{thm:CVXRS}, $f(x,u)=0$ and $h(x,u)\leq0$. 
To prove that this is a necessary condition, suppose $x$ and $u$ satisfy $f(x,u)=0$ and $h(x,u)\leq0$. Choose $z^u=z^\ell=x$, then it satisfies \cref{eqn:cvx_equiv}, and thus is feasible to \cref{eqn:conv_restr_eq}.
\end{proof}

\cref{corol:equiv_cvx} shows that the convex restriction can retrieve the original feasibility set if the original set is convex. If the feasibility set is non-convex, the convex restriction fundamentally cannot be equivalent to the feasibility set.
Note that the condition in \cref{lemma:iff_solvable} was a necessary and sufficient condition for feasibility, and there were two main steps that introduced conservatism of the convex restriction relative to $\mathcal{U}$. First is the tightness of the concave envelope. This is an unavoidable limitation where the nonlinear functions have to be bounded by concave envelopes. Second is the decomposition of the basis functions, and the use of the fact that the maximum of the sum is always less than the sum of the maximum,
\begin{displaymath}
\begin{aligned}
\max_{x\in\mathcal{P}(b)} K_i^+g^u(Cx,u)\leq K_i^+\max_{x\in\mathcal{P}(b)}g^u(Cx,u). \end{aligned}
\end{displaymath}
The more variables each combination of $g_i$ and $g_j$ share, the less conservative the convex restriction will be, but the complexity of the restriction will increase as it increases the degree of the sparsity $\lvert\mathcal{I}\rvert$. The next example shows this relationship more explicitly.

\begin{example}{(Conservatism v.s. Complexity Trade-off)} \label{ex:consv_vs_cplx}
Consider the following system of polynomial equations,
\begin{equation*} \begin{aligned}
x_1x_2+...+x_1x_n+u_1&=0 \\
x_i+u_i&=0,\ \  i=2,...,n \\
x_1x_2+...+x_1x_n&\leq10.
\end{aligned} \end{equation*}
For a given $k$, let us select the basis function to be
\begin{equation*}
\psi^{(k)}(z,u)=\begin{bmatrix} \sum_{i=2}^kx_1x_i,& x_1x_{k+1}, & ... & x_1x_n, & x^T, & u^T\end{bmatrix}^T,
\end{equation*}
with $z=x$. Decreasing $k$ decomposes the representation further and leads to a more sparse representation.
\cref{fig:consv_vs_cplx} shows the trade-off between the conservatism and the complexity as $k$ varies. 
The conservatism was quantified by solving $\min_{u\in\mathcal{U}^{\textrm{cvxrs},(k)}_{(0)}} u_1-u_1^*$ where $\mathcal{U}^{\textrm{cvxrs},(k)}_{(0)}$ is the convex restriction constructed with the basis function $\psi^{(k)}$ and $u_1^*=-10$ is the global optimal value. 
The complexity was quantified by the number of constraints involved in the convex restriction, which is proportional to $2^{k}$ for a naive implementation without vertex pruning. 
The degree of sparsity for $\psi^{(k)}_1=\langle x_1,\sum_{i=2}^kx_i\rangle$ is $\mathcal{I}_1=k$, and the vertex tracking require all combinations of $x_i\in\{x_i^u,\, x_i^\ell\}$ for $i=1,...,k$. However, this is a special case where the vertex pruning drastically reduces the number of constraints regardless of $k$. The maximum of $\psi^{(k)}_1$ occurs at $(x^u_1,\sum_{i=2}^kx^u_i)$ or $(x^\ell_1,\sum_{i=2}^kx^\ell_i)$ and only 2 vertices need to be tracked instead of $2^k$ vertices, and the restriction can scale without sacrificing the performance in this example.

\begin{figure}[!htbp]
	\centering
	\includegraphics[width=2.1in]{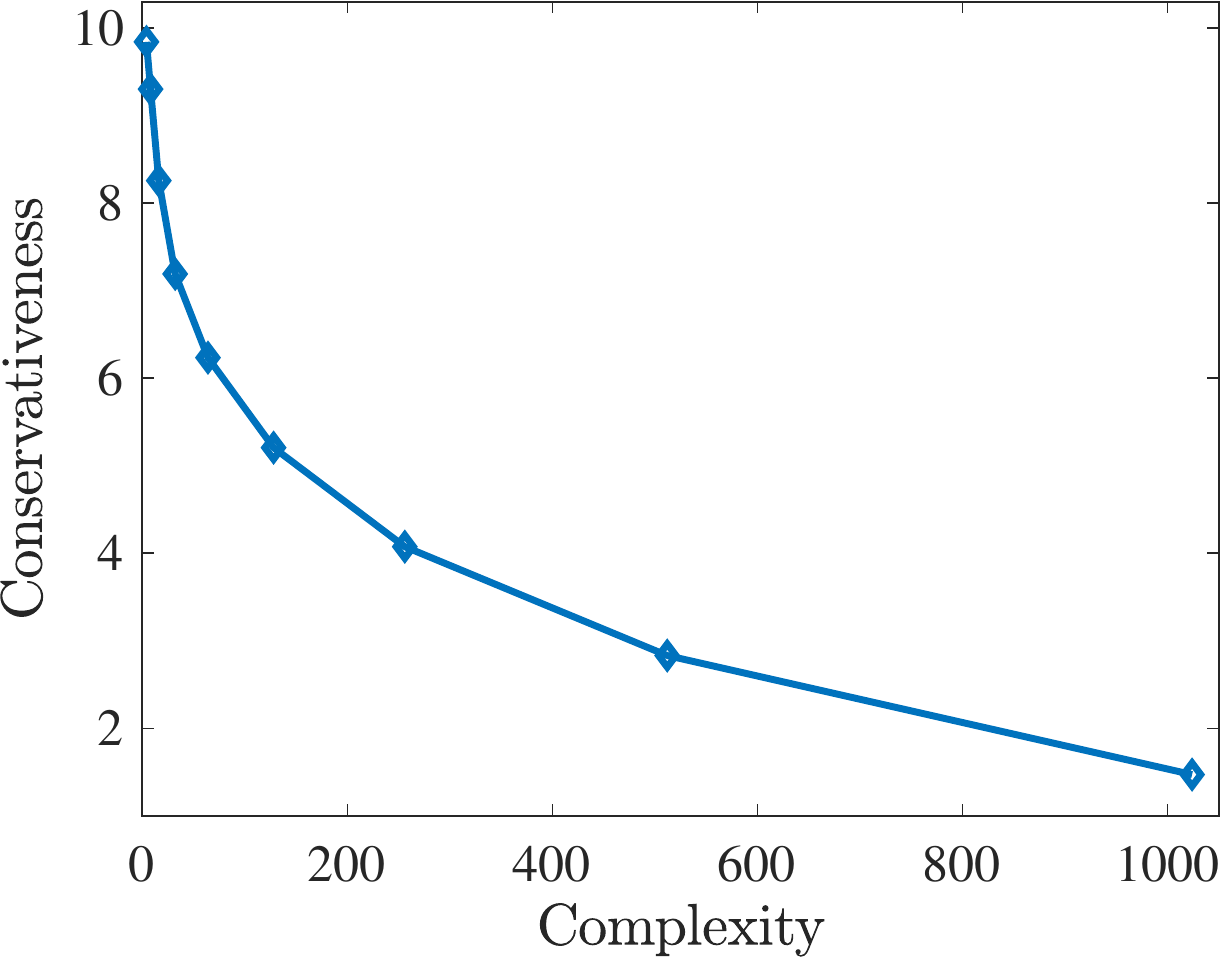}
	\caption{Illustration of the trade-off between the complexity and the conservatism. The complexity is quantified by the number of constraints involved, and the conservatism is quantified by the optimality gap.}
	\label{fig:consv_vs_cplx}
\end{figure}
\end{example}

\section{Convex Restriction under Uncertainty}\label{sec:rscrs}
In this section, we extend the convex restriction to include uncertain variables that are bounded by a given uncertainty set $\mathcal{W}\subseteq\mathbf{R}^r$.
We will assume that there is some known nominal value of the uncertain variable, which will be denoted by $w^{(0)}$.
The set of robust feasible explicit variables is denoted by
\begin{displaymath}
\mathcal{U}_\mathcal{W}=\{u\mid \forall w\in\mathcal{W},\, \exists\, x,\ f(x,u,w)=0,\ h(x,u,w)\leq 0  \}.
\end{displaymath}

\subsection{General Nonlinear Constraints}
The idea remains the same as the nominal constraint, and the only modification is that the concave envelopes need to capture the uncertainty set.
Similar to the previous section, the nonlinear functions are expressed by a linear combination of basis functions,
\begin{subequations} \begin{align}
f(x,u,w)&=M\psi(z,u,w) \label{eqn:robust_eq}\\
h(x,u,w)&=L\psi(z,u,w) \label{eqn:robust_ineq}.
\end{align}\end{subequations}
Equation \cref{eqn:robust_eq} can be written in the fixed point form,
\begin{displaymath}
    x=-(M\Lambda C)^{-1}Mg(z,u,w),
\end{displaymath} 
where $g(z,u,w)=\psi(z,u,w)-\Lambda z$. The matrix $\Lambda$ is chosen in the same way as \cref{subsec:Fixed_pt_eqn}, which is $\Lambda=\nabla_z \psi(z,u^{(0)},w^{(0)})\mid_{z=z^{(0)}}$ for differentiable $f$.
Let the nonlinear residual be bounded by
\begin{equation}
g^\ell_{\mathcal{W},k}(z,u)\leq g_k(z,u,w)\leq g^u_{\mathcal{W},k}(z,u), \ \forall w\in\mathcal{W},
\label{eqn:robust_envelope} \end{equation}
where $g^u_\mathcal{W}$ is a convex over-estimator and $g^\ell_\mathcal{W}$ is a concave under-estimator of $g$ over the uncertainty set $\mathcal{W}$.
Note that when we introduce uncertainty, we cannot satisfy $\cref{cond:env}$ (ii) and (iii) for the basis functions that are dependent on the uncertain variable.
Then the bounds over $\mathcal{P}$ and $\mathcal{W}$ can be expressed as

\begin{displaymath}\begin{aligned}
g^u_{\mathcal{PW},k}(u,b)&\geq g^u_{\mathcal{W},k}(z,u) & \forall z&\in\{z\mid z_i\in\{z^\ell_i,z^u_i\},\, \forall i\in \mathcal{I}_k\} \\
g^\ell_{\mathcal{PW},k}(u,b)&\leq g^\ell_{\mathcal{W},k}(z,u) & \forall z&\in\{z\mid z_i\in\{z^\ell_i,z^u_i\},\,\forall i\in \mathcal{I}_k\},
\end{aligned}\end{displaymath}
where the subscript $\mathcal{PW}$ indicates that it is a valid bound over the self-mapping polytope and the uncertainty set. Given these definitions, the following theorem provides a robust feasibility condition.

\begin{theorem}{(Robust Feasibility under General Uncertainty)}
For a given explicit variable $u$, there exists an implicit variable $x$ that satisfies $f(x,u,w)=0$ and $h(x,u,w)\leq0$ for all $w\in\mathcal{W}$ if there exists
$b\in\mathbf{R}^{2q}$ such that
\begin{subequations}\begin{align}
K^+g^u_\mathcal{PW}(u,b)+K^- g^\ell_\mathcal{PW}(u,b) &\leq b \label{eqn:robust1_conv_restr_eq1}
\\
L^+\psi^u_\mathcal{PW}(u,b)+L^-\psi^\ell_\mathcal{PW}(u,b) &\leq 0.
\label{eqn:robust1_conv_restr_eq2}
\end{align}\label{eqn:robust1_conv_restr_eq}\end{subequations}
\label{thm:robust1_feasibility}
\end{theorem}
\begin{proof}
The proof remains mostly similar to \cref{thm:CVXRS}. The condition \cref{eqn:robust1_conv_restr_eq} ensures that
\begin{displaymath}
\begin{aligned}
    \sup_{w\in\mathcal{W}}\max_{x\in\mathcal{P}(b)} K_ig(x,u,w)&\leq K_i^+\max_{x\in\mathcal{P}(b)}g^u_\mathcal{W}(x,u)+K_i^-\min_{x\in\mathcal{P}(b)}g^\ell_\mathcal{W}(x,u) \\
    &\leq K_i^+g^u_\mathcal{PW}(u,b)+K_i^- g^\ell_\mathcal{PW}(u,b)\leq b_i.
\end{aligned}
\end{displaymath}
From  \cref{lemma:iff_solvable}, there exists an implicit variable $x\in\mathcal{P}(b)$ for all $w\in\mathcal{W}$.
Similarly,
\begin{displaymath}
\begin{aligned}
    \sup_{w\in\mathcal{W}}\max_{x\in\mathcal{P}(b)} L_i\psi(x,u)&\leq L_i^+ \psi^u_\mathcal{PW}(u,b)+L_i^-\psi^\ell_\mathcal{PW}(u,b)\leq 0,\ \ i=1,\ldots,s,
\end{aligned}
\end{displaymath}
so for all $x\in\mathcal{P}(b)$ and $w\in\mathcal{W}$, $L\psi(Cx,u,w)\leq 0$. Therefore, there exists an implicit variable satisfying $f(x,u,w)=0$ and $h(x,u,w)\leq0$ for all $w\in\mathcal{W}$. 
\end{proof}

The convex restriction under uncertainty will be denoted by $\mathcal{U}^\textrm{cvxrs}_{\mathcal{W},(0)}\subseteq\mathcal{U}_\mathcal{W}$ where
\begin{displaymath}\begin{aligned}
\mathcal{U}^\textrm{cvxrs}_{\mathcal{W},(0)}=\{u\mid \exists\, b,\ K^+ g^u_\mathcal{PW}(u,b)+K^- g^\ell_\mathcal{PW}(u,b)&\leq b,\\
L^+ \psi^u_\mathcal{PW}(u,b)+L^- \psi^\ell_\mathcal{PW}(u,b) &\leq 0\}.
\end{aligned}\end{displaymath}
The subscript $(0)$ again indicates that the nominal point is $(x^{(0)},u^{(0)},w^{(0)})$.

When the explicit variables are given and the uncertainties are introduced, there will be generally a set of implicit variables defined through the realizations of the uncertain variable and the nonlinear equality constraints. The following remark shows the motivation and the role of the self-mapping polytope, which provides a bound on the set of implicit variables.
\begin{remark}
Given $u$ and $b$ satisfying the condition \cref{eqn:robust1_conv_restr_eq}, the self-mapping polytope, $\mathcal{P}(b)$, is an outer-approximation of all possible solutions for implicit variables under the uncertainty set $\mathcal{W}$.
\end{remark}

This gives an intuitive reason behind the convex restriction condition in the lifted domain with the parameter $b$, representing the bound on the implicit variables.
The following is an example of the envelopes that capture the uncertain variables.
\begin{example}{(Nonlinear Network Flow Problem under Uncertainty)}
Consider a special case of \cref{ex:oscillators} where the nonlinear flow models are subject to uncertainty,
\begin{displaymath}
\sigma(x_i-x_j)=w\sin(x_i-x_j),
\end{displaymath}
where the line parameter is subject to the uncertain variable $w$, bounded by $\mathcal{W}=\{w\mid w\in[w^\ell,\,w^u]\}$. Given that the basis functions are the same as \cref{ex:oscillators}, the residual function is
\begin{displaymath}
g_i(z_i,w)=w\sin{z}-w^{(0)}\cos{(z^{(0)})}z,
\end{displaymath}
for $i=1,\ldots,p$. The concave envelope that encloses the uncertainty set is then
\begin{displaymath} \begin{aligned}
g^u_{\mathcal{W},k}(z,u)&\geq\tilde{w}\sin{z_i^{(0)}}+\tilde{w}\cos{z_i^{(0)}}(z_i-z_i^{(0)})+\frac{\tilde{w}}{2}(z_i-z_i^{(0)})^2-w^{(0)}\cos{(z_i^{(0)})}z_i \\
g^\ell_{\mathcal{W},k}(z,u)&\leq\tilde{w}\sin{z_i^{(0)}}+\tilde{w}\cos{z_i^{(0)}}(z_i-z_i^{(0)})-\frac{\tilde{w}}{2}(z_i-z_i^{(0)})^2-w^{(0)}\cos{(z_i^{(0)})}z_i,
\end{aligned} \end{displaymath}
for $\tilde{w}\in\{w^u,\,w^\ell\}$. The convex restriction with the uncertain variable can be derived by \cref{thm:robust1_feasibility} using the envelope above.
\end{example}

Although this procedure is able to capture general nonlinearity and uncertainty sets, finding the concave envelope in \cref{eqn:robust_envelope} could be difficult for some of the applications. The next section discusses a special class of constraints where the robustness can be incorporated systematically.

\subsection{State-Uncertainty Separable Constraints}
In this section, we study a special case where the basis functions can be expressed by a sum of two nonlinear functions where implicit variables, $x$, and uncertain variables, $w$, are separable. Consider
\begin{equation} \begin{aligned}
f(x,u,w)&=M[\psi(x,u)+\alpha(u,w)] \\
h(x,u,w)&=L[\psi(x,u)+\beta(u,w)],
\end{aligned}\label{eqn:state_uncertainty_sepeqn}\end{equation}
where $\alpha:(\mathbf{R}^m,\mathbf{R}^r)\rightarrow \mathbf{R}^p$ and $\beta:(\mathbf{R}^m,\mathbf{R}^r)\rightarrow \mathbf{R}^p$ are vectors of continuous functions.
The functions $\alpha_i$ are linear with respect to $w$, and $L_j\beta$ are concave with respect to $w$ for all $u\in\mathbf{R}^m$. The uncertainty set $\mathcal{W}$ is a given non-empty, convex and compact set.
The derivation here closely follows \cite{Ben-Tal2015}, which provides a systematic way to construct the robust counterpart for nonlinear uncertain inequality constraints. Let us denote the convex conjugate of some function $\varphi$ as
\begin{displaymath}
\varphi^*(v)=\sup_{w\in\mathbf{dom}(\varphi)}\{v^Tw-\varphi(w)\},
\end{displaymath}
and the concave conjugate of $\varphi$ as
\begin{displaymath}
\varphi_*(v)=\inf_{w\in\mathbf{dom}(\varphi)}\{v^Tw-\varphi(w)\}.
\end{displaymath}
The indicator function of the set $\mathcal{W}$ is
\begin{displaymath}
\delta(w\mid \mathcal{W})=\begin{cases}0 & \text{if $w\in \mathcal{W}$} \\ \infty & \text{otherwise.} \end{cases}
\end{displaymath}
The support function of $\mathcal{W}$ is the conjugate of the indicator function,
\begin{equation}
\delta^*(v\mid \mathcal{W})=\sup_{w\in\mathbf{R}^r}\{v^Tw-\delta(w\mid \mathcal{W})\}=\sup_{w\in \mathcal{W}}v^Tw.
\label{eqn:supportf}\end{equation}

When the implicit variables and the uncertain variables are separable, there is a systematic way to derive the robust feasible condition using the support function and the conjugate function.

\begin{theorem}{(Robust Feasibility for State-Uncertainty Separable Constraints)}
For a given explicit variable $u$, there exists an implicit variable $x$ that satisfies constraints $f(x,u,w)=0$ and $h(x,u,w)\leq0$ for all $w\in\mathcal{W}$ if there exists
$b\in\mathbf{R}^{2q}$, $v\in\mathbf{R}^r$, and $y\in\mathbf{R}^r$ such that
\begin{subequations}\begin{align}
K^+g^u_\mathcal{P}(u,b)+K^- g^\ell_\mathcal{P}(u,b)+\xi(u,v) &\leq b \label{eqn:robust2_conv_restr_eq1} \\
L^+\psi^u_\mathcal{P}(u,b)+L^-\psi^\ell_\mathcal{P}(u,b)+\zeta(u,y) &\leq 0,
\label{eqn:robust2_conv_restr_eq2}
\end{align}\label{eqn:robust2_conv_restr_eq}\end{subequations}
where $\xi:(\mathbf{R}^m,\mathbf{R}^r)\rightarrow \mathbf{R}^{2q}$ and $\zeta:(\mathbf{R}^m,\mathbf{R}^r)\rightarrow \mathbf{R}^s$ are given by
\begin{equation}\begin{aligned}
\xi_i(u,v)&=\delta^*(v\mid \mathcal{W})-[K_i\alpha]_*(u,v) \\
\zeta_j(u,y)&=\delta^*(y\mid \mathcal{W})-[L_j\beta]_*(u,y).
\end{aligned}\label{eqn:xi_zeta_robust}\end{equation}
\label{thm:robust2_feasibility}
\end{theorem}
\begin{proof}
From the definition of indicator functions and using the Fenchel duality \cite{Bertsekas1999},
\begin{displaymath} \begin{aligned}
\max_{w\in\mathcal{W}}K_i\alpha(u,w)&=\max_{w\in\mathbf{R}^r}\left\{K_i\alpha(u,w)-\delta(w\mid\mathcal{W})\right\} \\
&=\min_{v\in\mathbf{R}^r}\left\{\delta^*(v\mid\mathcal{W})-[K_i\alpha]_*(u,v)\right\}.
\end{aligned} \end{displaymath}
Then using the expression above,
\begin{displaymath} \begin{aligned}
    \max_{w\in\mathcal{W}}\max_{x\in\mathcal{P}(b)} \left[K_ig(Cx,u)+K_i\alpha(u,w)\right] &\leq \max_{x\in\mathcal{P}(b)} K_ig(Cx,u)+\max_{w\in\mathcal{W}}K_i\alpha(u,w) \\    
    &\leq K_i^+g^u_\mathcal{P}(u,b)+K_i^- g^\ell_\mathcal{P}(u,b)+\xi_i(v,u)\leq b_i,
\end{aligned} \end{displaymath}
for some $v\in\mathbf{R}^r$. Therefore, the existence of $v\in\mathbf{R}^r$ guarantees the existence of an implicit variable under all realizations of $w\in\mathcal{W}$. Similarly, for $j=1,\ldots,s$,
\begin{displaymath}\begin{aligned}
    \sup_{w\in\mathcal{W}}\max_{x\in\mathcal{P}(b)} [L_j\psi(Cx,u)+L_j\beta(Cx,u)]&\leq L_j^+ \psi^u_\mathcal{P}(u,b)+L_j^-\psi^\ell_\mathcal{P}(u,b)+\zeta_j(u,y)\leq 0,
\end{aligned}\end{displaymath}
therefore, there exists an implicit variable satisfying $f(x,u,w)=0$ and $h(x,u,w)\leq0$ for all $w\in\mathcal{W}$. 
\end{proof}

There is a table of closed-form expressions for $\xi$ and $\zeta$ in \cref{eqn:xi_zeta_robust} depending on the function and the uncertainty set. We refer readers to \cite{Ben-Tal2015} for those cases, and we will show only one special case where those functions are linear with respect to $w$.

\subsubsection{Additive Uncertainty Constraints}
We consider again a special case of state-uncertainty separable constraints in \cref{eqn:state_uncertainty_sepeqn} where $\alpha$ and $\beta$ are linear functions of $w$ such that
\begin{equation} \begin{aligned}
f(x,u,w)&=M[\psi(x,u)+B w] \\
h(x,u,w)&=L[\psi(x,u)+D w],
\end{aligned}\end{equation}
where $B\in\mathbf{R}^{n\times r}$ and $D\in\mathbf{R}^{s\times r}$ are constant matrices. In addition, the uncertainty sets considered here are norm-bounded uncertainty sets,
\begin{equation}
\begin{aligned}
\mathcal{W}^Q(\gamma)&=\left\{w\mid \lVert w-w^{(0)} \rVert_2\leq \gamma\right\} \\
\mathcal{W}^B(\gamma)&=\left\{w\mid \lVert w-w^{(0)} \rVert_\infty\leq \gamma \right\},
\end{aligned}
\end{equation}
where $\gamma\in\mathbf{R}$ represents the margin. As $\gamma\rightarrow 0$, the uncertainty set vanishes, and the analysis on the nominal constraints applies. Moreover, there is the following manipulation to convert any general nonlinear uncertainty into additive uncertainty.

\begin{remark}
Any nonlinear constraint, $f(x,u,w)=0$ and $h(x,u,w)\leq0$, can be replaced with the additive uncertainty representation, $\tilde{f}(\tilde{x},u,w)=0$ and $\tilde{h}(\tilde{x},u,w)\leq0$. The functions $\tilde{f}:(\mathbf{R}^{n+r},\mathbf{R}^m,\mathbf{R}^r)\rightarrow \mathbf{R}^{n+r}$ and $\tilde{h}:(\mathbf{R}^{n+r},\mathbf{R}^m,\mathbf{R}^r)\rightarrow \mathbf{R}^s$ are
\begin{equation}
\tilde{f}(\tilde{x},u,w)=\begin{bmatrix} f(x,u,x_w) \\ x_w-w \end{bmatrix},\ \tilde{h}(\tilde{x},u,w)=h(x,u,x_w)
\end{equation}
where $x_w\in\mathbf{R}^r$ and $\tilde{x}=\begin{bmatrix}x^T & x_w^T\end{bmatrix}^T$. The replaced condition is equivalent to the original constraint, and the uncertainty $w$ enters the nonlinear equation as an additive term.
\end{remark}

When the system of nonlinear equations can be represented with the additive uncertainty, the following theorem provides a sufficient condition for robust feasibility.

\begin{corollary}{(Robust Feasibility for Additive Uncertainty)}
Suppose that the uncertainty set is given by a norm-bounded set, $\mathcal{W}(\gamma)=\{w\mid\lVert w\rVert\leq\gamma\}$.
For a given explicit variable $u$, there exists an implicit variable $x$ that satisfies $f(x,u,w)=0$ and $h(x,u,w)\leq0$ for all $w\in\mathcal{W}(\gamma)$ if there exists
$b\in\mathbf{R}^{2q}$ such that
\begin{subequations}\begin{align}
K^+g^u_\mathcal{P}(u,b)+K^- g^\ell_\mathcal{P}(u,b)+\xi(\gamma) &\leq b \\
L^+\psi^u_\mathcal{P}(u,b)+L^-\psi^\ell_\mathcal{P}(u,b)+\zeta(\gamma) &\leq 0,
\end{align}\end{subequations}
where for $i=1,\ldots,n$ and $j=1,\ldots,s$, $\xi_i$ and $\zeta_j$ are given by the following table depending on the type of uncertainty set.
\vskip0.7em
\begin{center} \begin{tabular}{|c|c|c|} \hline
 & $\mathcal{W}^Q(\gamma)$ & $\mathcal{W}^B(\gamma)$ \\ \hline
$ \xi_i(\gamma)$ & $K_iBw^{(0)}+\gamma\lVert K_iB\rVert_2$ & $K_iBw^{(0)}+\gamma\lVert K_iB\rVert_\infty$ \\ \hline
 $\zeta_j(\gamma)$ & $L_jDw^{(0)}+\gamma\lVert L_jD\rVert_2$ & $L_jDw^{(0)}+\gamma\lVert L_jD\rVert_\infty$ \\ 
\hline\end{tabular}\end{center}
\vskip0.7em
\label{corollary:addtive_uncertainty}\end{corollary}
\begin{proof}
This is a special case of \cref{thm:robust2_feasibility} with $\alpha(u,w)=Bw$ and $\beta(u,w)=Bw$, so Equation \cref{eqn:xi_zeta_robust} can be used to compute $\xi$ and $\zeta$.  Since $\alpha$ and $\beta$ are linear functions with respect to $w$, their concave conjugate functions are $[K_i\alpha]_*(u,v)=0$ with $v=(K_iB)^T$, and $[L_j\beta]_*(u,y)=0$ with $y=(L_jD)^T$. Substituting $v$ and $y$ to the support function,
\begin{displaymath}\begin{aligned}
\xi_i(\gamma)&=\delta^*(v\mid \mathcal{W}^Q(\gamma))\mid_{v=(K_iB)^T}=K_iBw^{(0)}+\gamma\lVert K_iB\rVert_2 \\
\zeta_j(\gamma)&=\delta^*(y\mid \mathcal{W}^Q(\gamma))\mid_{y=(L_jD)^T}=L_jDw^{(0)}+\gamma\lVert L_jD\rVert_2.
\end{aligned}\end{displaymath}
Similarly, the margins $\xi$ and $\zeta$ can be derived for the uncertainty set $\mathcal{W}^B(\gamma)$.
\end{proof}

Here the size of the uncertainty set is parametrized by $\gamma$ where the larger the $\gamma$, the more robust the system is against the uncertain variable. 
The robustness of a solution $(x^{(0)},u^{(0)})$ is often defined as how much uncertainty a solution can tolerate without violating the constraints.
With convex restriction and additive uncertainty constraints, the lower bound on the margin can be computed by solving a convex optimization problem.

\begin{corollary}{(Robustness Margin)}
Suppose that $\gamma\in\mathbf{R}$ is given by solving the following optimization problem,
\begin{equation}\begin{aligned}
	\underset{\gamma,b}{\text{maximize}} \hskip 1em & \gamma \\
	\text{subject to} \hskip 1em & K^+ g^u_\mathcal{P}(u^{(0)},b)+K^- g^\ell_\mathcal{P}(u^{(0)},b)+\xi(\gamma)\leq b \\
	& L^+ \psi^u_\mathcal{P}(u^{(0)},b)+L^- \psi^\ell_\mathcal{P}(u^{(0)},b)+\zeta(\gamma) \leq 0,
\end{aligned}\label{eqn:robOPT_gamma}\end{equation}
where $\xi$ and $\zeta$ are linear functions of $\gamma$ given in \cref{corollary:addtive_uncertainty}. Then the explicit variable $u^{(0)}$ has a corresponding implicit variable $x$ satisfying $f(x,u,w)=0$ and $h(x,u,w)\leq0$ for all realizations of the uncertainty set $\mathcal{W}(\gamma)=\{w\mid \lVert w-w^{(0)}\rVert\leq\gamma\}$.
\label{thm:robust3_feasibility}
\end{corollary}

In addition to finding the robustness margin of a solution, the explicit variable $u^{(0)}$ can be iteratively updated to find the optimal solution given the nonlinear equality and inequality constraints.

\section{Sequential Convex Restriction}\label{sec:scrs}
In this section, we develop the algorithm to solve the robust optimization in \cref{eqn:robustOPT_orig},

\begin{displaymath}\begin{aligned}
		\underset{u}{\text{minimize}} \hskip 1em & f_0(u) \\
		\text{subject to} \hskip 1em & \forall w\in\mathcal{W},\ \exists\, x\in\mathbf{R}^n,\ f(x,u,w)= 0, \ h(x,u,w)\leq 0.
\end{aligned}\end{displaymath}

The non-convex constraints in this problem can be restricted to convex constraints by the conditions provided in \cref{thm:CVXRS} for the nominal constraints and \cref{thm:robust1_feasibility} for constraints including uncertainty. Special cases such as state-uncertainty separable constraints or additive uncertainty constraints can use the convex restrictions in \cref{thm:robust2_feasibility} and \cref{corollary:addtive_uncertainty}, respectively.
The problem can be solved by iterating between (a) solving the optimization with convex restriction, and (b) setting the solution as the new nominal point for constructing the convex restriction. The algorithm described here is named sequential convex restriction, and the procedure is described in \cref{alg:SCRS} with some termination thresholds $\varepsilon_1,\varepsilon_2,\varepsilon_3>0$.

\begin{algorithm}[!htbp]
\begin{algorithmic}
\STATE \textit{Initialization}: $u^{(0)}$, $x^{(0)}$, and $k=0$
\WHILE {$\lVert u^{(k+1)}-u^{(k)}\rVert_2>\varepsilon_1$ or $\lVert f_0(u^{(k+1)})-f_0(u^{(k)})\rVert_2>\varepsilon_2$}
\STATE $K=\begin{bmatrix}-I & I\end{bmatrix}^TC(M\Lambda C)^{-1}M$
\STATE $u^{(k+1)}=\argmin_{u\in\mathcal{U}^\textrm{cvxrs}_{\mathcal{W},(k)}} f_0(u)$
\STATE $x^{(k+1)}=x^{(k)}$
\WHILE{$\lVert f(x^{(k+1)},u^{(k+1)},w^{(0)})\rVert_2>\varepsilon_3$}
\STATE $x^{(k+1)}=-(M\Lambda C)^{-1}Mg(x^{(k+1)},u^{(k+1)},w^{(0)})$
\ENDWHILE
\STATE $k:=k+1$
\ENDWHILE
\end{algorithmic}
\caption{Sequential Convex Restriction}
\label{alg:SCRS}
\end{algorithm}

There are three computationally notable steps, which are computing the inverse of the Jacobian to compute $K$, solving the convex optimization problem with convex restriction, and retrieving the nominal implicit variable. The retrieval of the implicit variable leverages \cref{remark:retrieval_implicit} in the proposed algorithm, but this step can be replaced by other procedures such as Newton's method or the Gauss-Seidel method.

\subsection{Analysis on the Subproblems}
Sequential Convex Restriction solves the following convex optimization problem as the subproblems of the algorithm, 
\begin{equation}
	\begin{aligned}
		\underset{u,b}{\text{minimize}} \hskip 1em & f_0(u) \\
		\text{subject to} \hskip 1em & K^+ g^u_\mathcal{PW}(u,b)+K^- g^\ell_\mathcal{PW}(u,b)\leq b \\
		& L^+ \psi^u_\mathcal{PW}(u,b)+L^- \psi^\ell_\mathcal{PW}(u,b) \leq 0.
	\end{aligned}
	\label{eqn:OPT_cvxrs}
\end{equation}

The key feature of the convex restriction is that the non-convex constraint can be replaced by a convex approximation that guarantees a feasible solution. 
Moreover, the containment of the nominal point from \cref{lemma:non_empty} ensures that the optimal value is improved at every iteration.

\begin{corollary}{(Bounds on the Optimal Cost)}
Suppose that $u^{(k+1)}$ denotes the solution of the problem in \cref{eqn:OPT_cvxrs}:
\begin{equation}
u^{(k+1)}=\argmin_{u\in\mathcal{U}^\textrm{cvxrs}_{\mathcal{W},(k)}} f_0(u).
\end{equation}
The optimal value of the problem is bounded by
\begin{equation}
f_0(u^\textrm{opt})\leq f_0(u^{(k+1)})\leq f_0(u^{(k)}),
\end{equation}
where $u^\textrm{opt}$ is the global optimal solution of the problem in \cref{eqn:robustOPT_orig}.
\label{corollary:bound_opt}
\end{corollary}
\begin{proof}
The lower bound comes from the definition of the global optimal value. From \cref{lemma:non_empty}, the convex restriction always contains the nominal point, $u^{(k)}\in\mathcal{U}^\textrm{cvxrs}_{\mathcal{W},(k)}$. Therefore, $\min_{u\in\mathcal{U}^\textrm{cvxrs}_{\mathcal{W},(k)}} f_0(u)\leq f_0(u^{(k)})$.
\end{proof}

\subsection{Sequential Convex Restriction for Nominal Constraints}
In this section, we provide the analysis of the algorithm for solving the nominal problem,
\begin{equation}
\underset{u,x}{\text{minimize}} \hskip1em f_0(u), \hskip1em \text{subject to} \hskip1em f(x,u)= 0, \ h(x,u)\leq 0.
\label{eqn:nominalOPT}
\end{equation}

Sequential convex restriction (SCRS) for nominal constraints belongs to the family of Sequential Convex Optimization, which is a local search method that iteratively solves convex approximations of the original problem. In particular, the related classical algorithms are the Sequential Quadratic Programming (SQP) and trust region methods \cite{Boggs1995,conn2000trust,Gould2005,nocedal2006numerical}. 
While these methods showed success in practice for solving a large optimization with equality constraints, some of the possible shortcomings were (i) the linearized constraints may be inconsistent, (ii) the solution may be infeasible, and (iii) the iteration may diverge. These shortcomings could be overcome by using extended methods such as Inexact SQP \cite{Curtis2014,Byrd2010,Byrd2008}.
Sequential convex restriction provides potentially a more elegant way to handle the shortcomings of SQP.
An alternative view of SCRS is that the self-mapping set $\mathcal{P}(b)$ can be interpreted as a trust region, and the lifted formulation allows us to co-optimize the decision variables and the trust region cast as a single convex optimization problem.
Next, we show the convergence result for the algorithm, which states that the converged point will satisfy the KKT condition, which is a necessary condition for optimality for non-convex problems.

\begin{corollary}{(Convergence of SCRS)}
Suppose the explicit variable $u^*$ is the output of \cref{alg:SCRS} such that
\begin{equation}
u^*=\lim_{k\rightarrow\infty}\argmin_{u^{(k+1)}\in\mathcal{U}^\textrm{cvxrs}_{\mathcal{W},(k)}} f_0(u).
\end{equation}
Then, there exists a corresponding implicit variable $x^*$ such that $(x^*,u^*)$ is feasible and
\begin{itemize}
\item $\nabla_x f(x^*,u^*)\mid_{x=x^*}$ is singular, or 
\item $(x^*,u^*)$ satisfies the KKT condition of the original problem in \cref{eqn:nominalOPT}.
\end{itemize}
\label{corol:KKT}\end{corollary}
\begin{proof}
The algorithm yields a sequence of explicit variables $\{u^{(k)}\}$ that satisfies $f_0(u^{(k+1)})\leq f_0(u^{(k)})$ from \cref{corollary:bound_opt}. Moreover, since the sequence is bounded below by the global optimal solution $f_0^*$, the sequence converges to a finite value from the Monotone Convergence Theorem. Suppose the converged solution is denoted by $(x^*,u^*)$, which satisfies
\begin{equation}
u^*=\argmin_{u\in\mathcal{U}^\textrm{cvxrs}_*} f_0(u),
\label{eqn:u0_cvxrsopt}
\end{equation}
where $\mathcal{U}^\textrm{cvxrs}_*$ is the convex restriction with $(x^*,u^*)$ as the nominal point.
Without loss of generality, $b^*=\begin{bmatrix} (Cx^*)^T & -(Cx^*)^T \end{bmatrix}^T$ is always a feasible solution for $\mathcal{U}^\textrm{cvxrs}_*$ and is an optimal solution for the above problem. Suppose $\nabla_x f(x^*,u^*)\mid_{x=x^*}$ is non-singular, then $M\Lambda C$ is invertible. Let $\Gamma=-C(M\Lambda C)^{-1}M$, then the following KKT condition is a necessary and sufficient condition for optimality of the problem \cref{eqn:u0_cvxrsopt},
\begin{equation*}\begin{aligned}
\Gamma^+g^u_\mathcal{P}(u^*,b^*)+\Gamma^- g^\ell_\mathcal{P}(u^*,b^*) &\leq z^*, \\
-\Gamma^-g^u_\mathcal{P}(u^*,b^*)-\Gamma^+ g^\ell_\mathcal{P}(u^*,b^*) &\leq -z^*, \\
L^+\psi^u_\mathcal{P}(u^*,b^*)+L^-\psi^\ell_\mathcal{P}(u^*,b^*) &\leq 0, \\
\lambda^*_1\geq 0,\ \lambda^*_2\geq 0,\ \lambda^*_3&\geq 0, \\
\lambda^*_{1,i}\Gamma_i^+g^u_\mathcal{P}(u^*,b^*)+\lambda^*_{1,i}\Gamma_i^- g^\ell_\mathcal{P}(u^*,b^*) &= \lambda^*_{1,i}z_i^*,\ i=1,...,q, \\
-\lambda^*_{2,i}\Gamma_i^-g^u_\mathcal{P}(u^*,b^*)-\lambda^*_{2,i}\Gamma_i^+ g^\ell_\mathcal{P}(u^*,b^*) &= -\lambda^*_{2,i}z_i^*,\ i=1,...,q, \\
\lambda^*_{3,i}L_i^+\psi^u_\mathcal{P}(u^*,b^*)+\lambda^*_{3,i}L_i^-\psi^\ell_\mathcal{P}(u^*,b^*) &= 0,\ i=1,...,s, \\
\nabla f_0(u^*)+\sum_{i=1}^q\lambda_{1,i}^*\left\{\Gamma_i^+\nabla g^u_\mathcal{P}(u^*,b^*)+\Gamma_i^- \nabla g^\ell_\mathcal{P}(u^*,b^*)-\nabla z_i(z^*)\right\}\hskip -10em&\\
+\sum_{i=1}^q\lambda_{2,i}^*\left\{-\Gamma_i^-\nabla g^u_\mathcal{P}(u^*,b^*)-\Gamma_i^+ \nabla g^\ell_\mathcal{P}(u^*,b^*)+\nabla z_i(z^*)\right\}\hskip -10em&\\
+\sum_{i=1}^q\lambda_{3,i}^*\left\{L_i^+\nabla \psi^u_\mathcal{P}(u^*,b^*)+L_i^-\nabla \psi^\ell_\mathcal{P}(u^*,b^*)\right\}=0\hskip -10em&
\end{aligned}\end{equation*}
Since $(x^*,u^*)$ is the nominal point and satisfies \cref{cond:env},
\begin{displaymath}\begin{aligned}
g^\ell_\mathcal{P}(u^*,b^*)&=g(z^*,u^*)=g^u_\mathcal{P}(u^*,b^*), \\
\nabla g^\ell_\mathcal{P}(u^*,b^*)&=\nabla g(u^*,b^*)=\nabla g^u_\mathcal{P}(u^*,b^*).
\end{aligned}\end{displaymath}
Substitute the above equation and $\nu_i=\sum_{j=1}^q(\lambda_{2,j}^*-\lambda_{1,j}^*)C_{ji}$ for $i=1,...,n$ to the KKT condition of the problem \cref{eqn:u0_cvxrsopt}, then
\begin{equation*}\begin{aligned}
-(M\Lambda C)^{-1}Mg(x^*,u^*) = x^*, \
L\psi(x^*,u^*)&\leq 0,\\ 
\lambda^*_3\geq 0, \
\lambda^*_{3,i}L_i\psi(x^*,u^*) = 0,\ i=1,...,s,& \\
\nabla f_0(u^*)+\sum_{i=1}^n\nu_i^*\left\{[(M\Lambda C)^{-1}]_iM\nabla g(z^*,u^*)+\nabla x_i(x^*)\right\}+\sum_{i=1}^q\lambda_{3,i}^*L_i\nabla \psi(z^*,u^*)&=0
\end{aligned}\end{equation*}
which is the KKT condition of the nominal problem in \cref{eqn:nominalOPT} where the equality constraint is replaced by its fixed point representation.
\end{proof}

Next, we show an example in a polynomial optimization problem that includes nonlinear equality constraints.

\begin{example}{(Polynomial Optimization)} In this example, a polynomial optimization problem adapted from an example in \cite{Park2017} is considered,
\begin{equation*}
	\begin{aligned}
		\underset{u,x}{\text{minimize}} \hskip 1em & u_3 \\
		\text{subject to} \hskip 1em & x_1^2+x_2^2+x_3^2-1=0 \\
		& u_1-x_1^2+w_1=0 \\
		& u_2-x_2x_3+w_2=0 \\
		& x_1u_1-2x_1u_2+x_2\leq u_3,\ \forall w\in\mathcal{W}^Q(\gamma),
	\end{aligned}
\end{equation*}
where $\mathcal{W}^Q(\gamma)=\{w\mid \lVert w\rVert_2\leq\gamma\}$. In this example, we consider the nominal problem where $\gamma=0$ so that $w=0$. The robust optimization will be considered later in \cref{example:robust_poly_opt}. \cref{fig:scrs_poly} shows the convergence of the sequential convex restriction described in \cref{alg:SCRS} with four different initial conditions. The triangular-shaped feasible region is created by the solvability condition, and the convergence of the algorithm depends on the initialization. The global optimal point is achieved with the initial condition in (a) in this example. The initial conditions in (a) and (b) arrive at a local optimal point satisfying the KKT conditions. The initial condition in (c) arrives at the boundary of the constraints where $\nabla_x f(x,u^*)\mid_{x=x^*}$ becomes singular.
\begin{figure}[!htbp]
	\centering
	\includegraphics[width=4.5in]{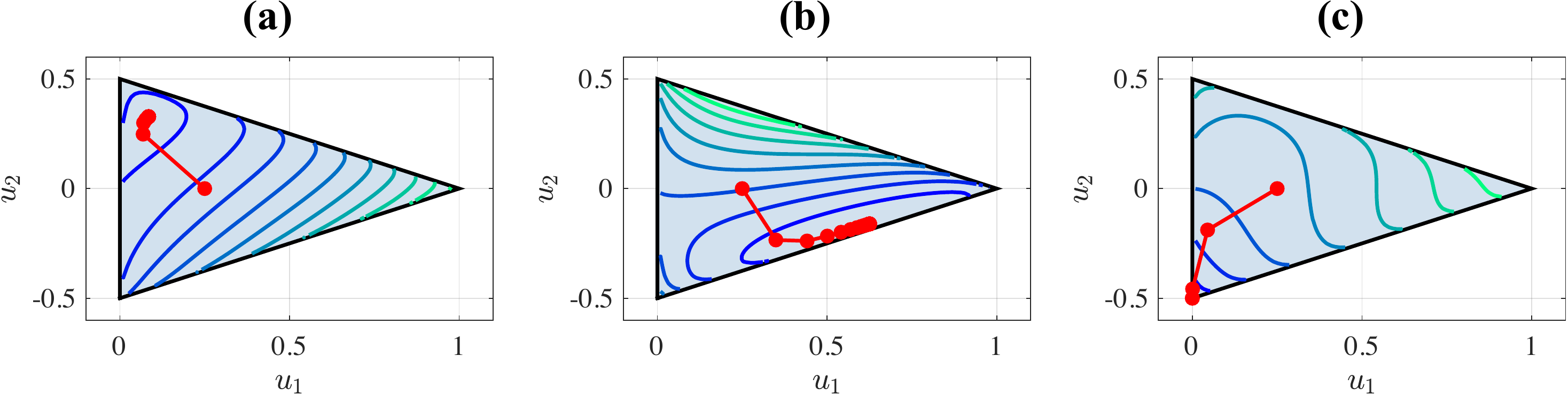}
	\caption{The blue region represents the feasible region and the contour line shows the objective function where the darker contour lines have the lower objective value. The initial condition for the explicit variable was set to $u^{(0)}=[0.25,\,0,\,2]$. The initial condition for the implicit variable was set to (a) $[0.5,\,-0.866,\,0]$, (b) $[-0.5,\,-0.866,\,0]$, and (c) $[0.5,\,0,\,0.866]$.}
	\label{fig:scrs_poly}
\end{figure}
\label{example:poly}
\end{example}

A larger example for solving the Optimal Power Flow problem using the sequential convex restriction was considered in \cite{Lee2019} as an extension of \cite{Lee2018} for a power systems application. In the next section, we will consider the optimization problem that includes bounded uncertain variables, which is the main motivation for using sequential convex restriction.

\subsection{Sequential Convex Restriction for Robust Optimization}
In this section, we extend the sequential convex restriction to solve the robust optimization problems with nonlinear equality constraints in \cref{eqn:robustOPT_orig}. 

Many classes of robust optimization problems are known to have counterparts that can be solved with a finite and explicit optimization problem, however, those results are limited to nonlinear inequality constraints \cite{Ben-Tal1999,Bertsimas2004,Ben-Tal2003,Averbakh2008,BenTal2009,Bertsimas2010,Beyer2007,Gabrel2014}.
The equality constraints were mostly assumed to be linear and studied under a special context \cite{ElGhaoui1997,Calayore2004}.
The equality constraint was considered in \cite{Zhang2007}, but it relies on the first-order approximation around its neighborhood and does not provide a rigorous guarantee.

On the other hand, sequential convex restriction described in \cref{alg:SCRS} gives a guarantee for robustness against the given uncertainty set.
Moreover, we discussed a number of results for the convex restriction of the nominal constraints, and these results imply that SCRS for robust optimization problem will yield a good approximate solution.
One thing to note is that while SCRS guarantees robustness, the optimality is not necessarily guaranteed. The following remark provides a practical way to quantify the optimality gap.

\begin{corollary}{(Optimality Gap for Robust Optimization Problem)}
Suppose that $u^*$ denotes the converged solution of \cref{alg:SCRS}.
The optimality gap can be bounded by
\begin{equation}
f_0(u^*)-f_0(u^\textrm{robust-opt})\leq f_0(u^*)-f_0(u^\textrm{nominal-opt}),
\label{eqn:robust_opt_gap}
\end{equation}
where $u^\textrm{robust-opt}$ is the globally optimal solution for the robust optimization problem in \cref{eqn:robustOPT_orig}, and $u^\textrm{nominal-opt}$ is the globally optimal solution of the nominal problem in \cref{eqn:nominalOPT}.
\end{corollary}
\begin{proof}
Since the nominal uncertainty is a special case in the uncertainty set, $w^{(0)}\in\mathcal{W}$, it follows that $f_0(u^\textrm{nominal-opt})\leq f_0(u^\textrm{robust-opt})$. A simple rearrangement leads to the condition in \cref{eqn:robust_opt_gap}.
\end{proof}

Next, we show an example in polynomial optimization subject to additive uncertainties.

\begin{example}\textbf{(Polynomial Optimization)}
Consider the robust optimization problem in \cref{example:poly} where the uncertainty set is $\mathcal{W}=\{w\mid\lVert w\rVert_2\leq\gamma\}$ with $\gamma>0$. In this example, the uncertainty is additive to the nonlinear equation, and the condition from \cref{corollary:addtive_uncertainty} was used to guarantee robustness.
\cref{fig:Rscrs_poly} shows the illustration of the results for various sizes of the uncertainty set and initial conditions.
\begin{figure}[!htbp]
	\centering
	\includegraphics[width=4.5in]{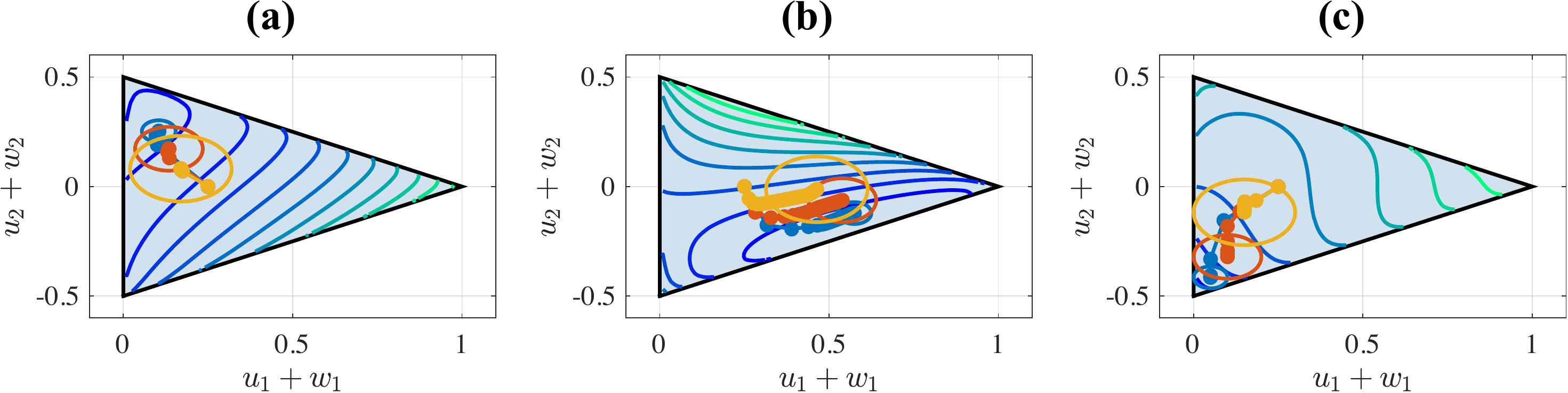}
	\caption{The convergence of sequential convex restriction with $\gamma=0.05$ (blue), $\gamma=0.1$ (red), and $\gamma=0.15$ (yellow).}
	\label{fig:Rscrs_poly}
\end{figure}
\label{example:robust_poly_opt}
\end{example}

\section{Concluding remarks}\label{sec:conclusion}
In this paper, we have developed the sequential convex restriction for solving a robust optimization problem with nonlinear equality and inequality constraints. We expand the convex restriction of nominal constraints and develop sufficient conditions for robust feasibility against the given uncertainty set. The algorithm guarantees robust feasibility of the solution at every iteration by leveraging the conditions from convex restriction. We showed that the algorithm asymptotically converges to a solution satisfying KKT condition for the nominal constraints. 

\appendix
\section{Concave Envelopes}\label{apdx:concave_env}

\subsection{Concave Envelope for Bilinear Function}\label{apdx:bilinear}
A bilinear function can be bounded by the following concave envelopes with some $\rho_1, \, \rho_2>0$ and the nominal point $(x^{(0)},\, y^{(0)})$ \cite{Lee2018},
\begin{equation}
\begin{aligned}
xy&\geq x^{(0)}y^{(0)}+y^{(0)}(x-x^{(0)})+x^{(0)}(y-y^{(0)})-\frac{1}{4}\left[\rho_1(x-x^{(0)})-\frac{1}{\rho_1}(y-y^{(0)})\right]^2 \\
xy&\leq x^{(0)}y^{(0)}+y^{(0)}(x-x^{(0)})+x^{(0)}(y-y^{(0)})+\frac{1}{4}\left[\rho_2(x-x^{(0)})+\frac{1}{\rho_2}(y-y^{(0)})\right]^2.
\end{aligned}
\end{equation}
The over-estimator is tight along $\rho_2(x-x^{(0)})-\frac{1}{\rho_2}(y-y^{(0)})=0$, and the under-estimator is tight along $\rho_2(x-x^{(0)})+\frac{1}{\rho_2}(y-y^{(0)})=0$. Both over- and under-estimators are tight at the nominal point, $(x^{(0)},y^{(0)})$.

\begin{figure}[!htbp]
	\centering
	\includegraphics[width=2.1in]{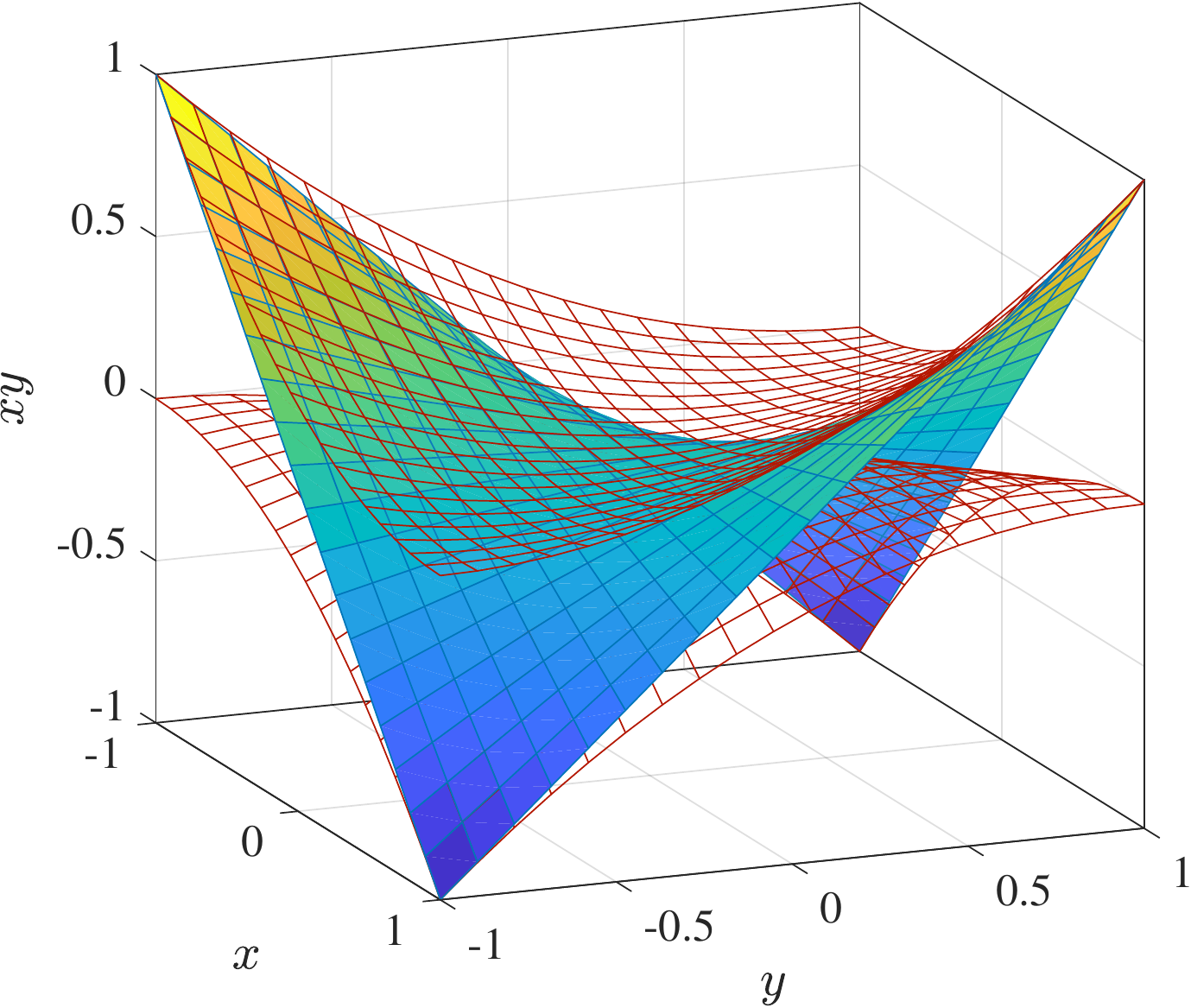}
	\caption{Illustration of concave envelopes for a bilinear function.}
	\label{fig_convenv_ex}
\end{figure}

\subsection{Concave Envelope for Unitary Quadratic Function}
A unitary quad-ratic function can be bounded by the following quadratic concave envelopes for all $x$ given the nominal point $x^{(0)}$ \cite{Lee2018},
\begin{equation}
\begin{aligned}
x^2&\geq (x^{(0)})^2+2x^{(0)}(x-x^{(0)})+(x-x^{(0)})^2=x^2 \\
x^2&\leq (x^{(0)})^2+2x^{(0)}(x-x^{(0)})=2x^{(0)}x-(x^{(0)})^2.
\end{aligned}
\end{equation}

\subsection{Concave Envelope for Trigonometric Function}
Trigonometric functions can be bounded by the following quadratic concave envelopes for all $\theta$ given the nominal point $\theta_0$,
\begin{displaymath}
\begin{aligned}
\sin\theta&\geq\sin{\theta^{(0)}}+\cos{\theta^{(0)}}(\theta-\theta^{(0)})-\frac{1}{2}(\theta-\theta^{(0)})^2 \\
\sin\theta&\leq \sin{\theta^{(0)}}+\cos{\theta^{(0)}}(\theta-\theta^{(0)})+\frac{1}{2}(\theta-\theta^{(0)})^2,
\end{aligned}
\end{displaymath}
\begin{displaymath}
\begin{aligned}
\cos\theta&\geq\cos{\theta^{(0)}}-\sin{\theta^{(0)}}(\theta-\theta^{(0)})-\frac{1}{2}(\theta-\theta^{(0)})^2 \\
\cos\theta&\leq \cos{\theta^{(0)}}-\sin{\theta^{(0)}}(\theta-\theta^{(0)})+\frac{1}{2}(\theta-\theta^{(0)})^2.
\end{aligned}
\end{displaymath}

\subsection{Concave Envelope for Logistic Function}
A logistic function, $\sigma(x)=\frac{1}{1+e^{-x}}$, has the bounded second derivative of $\frac{\sqrt{3}}{18}$, and its quadratic concave envelope is
\begin{equation*}
\begin{aligned}
\sigma(x)&\geq\sigma^{(0)}+\sigma^{(0)}(1-\sigma^{(0)})(x-x^{(0)})-\frac{\sqrt{3}}{36}(x-x^{(0)})^2 \\
\sigma(x)&\leq \sigma^{(0)}+\sigma^{(0)}(1-\sigma^{(0)})(x-x^{(0)})+\frac{\sqrt{3}}{36}(x-x^{(0)})^2.
\end{aligned}
\end{equation*}
where $x^{(0)}$ is the nominal point and $\sigma^{(0)}=\frac{1}{1+e^{-x^{(0)}}}$.

\bibliographystyle{siam}
\bibliography{references}
\end{document}